\documentclass[12pt]{amsart}
\usepackage{amsmath}
\usepackage{amssymb}
\usepackage{amstext}
\usepackage{amscd}
\usepackage{amsthm}
\usepackage{hyperref}
\usepackage[all]{xy}
\usepackage{tikz-cd}

\usepackage[margin=2.5cm]{geometry}

\newtheorem{teo}{Theorem}
\newtheorem{prop}[teo]{Proposition}
\newtheorem{lem}[teo]{Lemma}

\newtheorem{conj}[teo]{Conjecture}

\newtheorem{rem}[teo]{Remark}

\newcommand{\Lie}{{\rm Lie}}

\newcommand{\der}{{\rm der}}

\newcommand{\ad}{{\rm ad}}

\newcommand{\CC}{{\mathbb C}}
\newcommand{\RR}{{\mathbb R}}
\newcommand{\ZZ}{{\mathbb Z}}
\newcommand{\QQ}{{\mathbb Q}}
\newcommand{\NN}{{\mathbb N}}

\newcommand{\PP}{{\mathbb P}}

\newcommand{\GG}{{\mathbb G}}
\newcommand{\SSS}{{\mathbb S}}
\newcommand{\AAA}{{\mathbb A}}

\newcommand{\OOO}{{\mathcal O}}

\newcommand{\Opt}{{\rm Opt}}

\newcommand{\cF}{{\mathcal F}}
\newcommand{\cA}{{\mathcal A}}

\newcommand{\cP}{{\mathcal P}}

\newcommand{\cS}{{\mathcal S}}
\newcommand{\cL}{{\mathcal L}}

\newcommand{\cR}{{\mathcal R}}

\newcommand{\gG}{{\bf G}}
\newcommand{\gX}{{\bf X}}
\newcommand{\gGL}{{\bf GL}}
\newcommand{\gGSp}{{\bf GSp}}
\newcommand{\gSp}{{\bf Sp}}

\newcommand{\gH}{{\bf H}}
\newcommand{\gN}{{\bf N}}
\newcommand{\gZ}{{\bf Z}}

\DeclareMathOperator{\mult}{mult}
\DeclareMathOperator{\codim}{codim}

\newcommand{\vxi}{{\boldsymbol\xi}}
\def\rest#1{{\vert_{#1}}} 

\newcommand{\mo}[1][k]{M^{\smash{(#1)}}}
\def\pd#1#2{\tfrac{\partial#1}{\partial#2}}

\def\A{{\mathbb{A}}}
\def\N{{\mathbb{N}}}
  
\title{Effective computations for weakly optimal subvarieties}

\author{Gal Binyamini}
\address{Binyamini: Weizmann Institute of Science, Rehovot, Israel}
\email{gal.binyamini@weizmann.ac.il}

\author{Christopher Daw}
\address{Daw: Department of Mathematics and Statistics, University of Reading,
    White\-knights,  PO Box 217,  Reading,  Berkshire RG6 6AH,  United Kingdom}
\email{chris.daw@reading.ac.uk}

\subjclass[2010]{}
\keywords{}

\begin{document}

\maketitle

\begin{abstract}

Ren and the second author established that the weakly optimal subvarieties (e.g. maximal weakly special subvarieties) of a subvariety $V$ of a Shimura variety arise in finitely many families. In this article, we refine this theorem by (1) constructing a finite collection of algebraic families whose fibers are precisely the weakly optimal subvarieties of $V$; (2) obtaining effective degree bounds on the weakly optimal locus and its individual members; (3) describing an effective procedure to determine the weakly optimal locus.
\end{abstract}

\setcounter{tocdepth}{1}
\tableofcontents

\section{Introduction}

This article is concerned with effective results on the (geometric side of the) Zilber--Pink conjecture for (pure) Shimura varieties. The conjecture itself is as follows.

\begin{conj}[Zilber--Pink]\label{ZP:intro}
Let $S$ be a Shimura variety and let $V$ be a Hodge generic, irreducible, algebraic subvariety of $S$. Then the intersection of $V$ with the union of the special subvarieties of $S$ of codimension at least $\dim V+1$ is not Zariski dense in $V$.
\end{conj}

By \cite[Theorem 12.4]{BD}, Conjecture \ref{ZP:intro} is equivalent to the seemingly stronger variations involving atypical intersections and optimal subvarieties, the latter of which states that $V$ contains only finitely many optimal subvarieties. We state this version and the necessary definitions in Section \ref{SSZP}.

In \cite{DR}, Ren and the second author outline a Pila--Zannier strategy for proving Conjecture \ref{ZP:intro}, applying the familiar combination of o-minimality, functional transcendence, and arithmetic (see, in particular, \cite{hp:o-min}). The unconditional aspect of that strategy is what might be considered the geometric Zilber--Pink conjecture for Shimura varieties. Gao \cite{Gao:Ax} refers to its generalization to mixed Shimura varieties as {\it a finiteness result \`a la Bogomolov}. The statement is as follows. We refer to Section \ref{wswo} for the relevant facts and definitions, recalling here only that an optimal subvariety is, in particular, a weakly optimal subvariety.

\begin{teo}[\cite{DR}, Proposition 6.3]
Let $S=\Gamma\backslash X$ be a Shimura variety and let $V$ be an irreducible algebraic subvariety of $S$. There exists a finite set $\Sigma$ of pairs $(X_\gH,X_1\times X_2)$ with $X_\gH$ a pre-special subvariety of $X$ and $X_\gH=X_1\times X_2$ a $\QQ$--splitting of $X_\gH$, such that, for any weakly optimal subvariety $W$ of $V$, the weakly special closure of $W$ in $S$ is equal to the image in $S$ of $X_1\times\{x_2\}$, for some $(X_\gH,X_1\times X_2)\in\Sigma$ and some $x_2\in X_2$.
\end{teo}

In this article, we give a refined version of this theorem (see
Theorem \ref{main-const-thm}). We summarize this refinement as follows. We denote by $P$ the so-called {\it standard principal bundle} associated with $S$ and,
by $\Omega$, a Chow variety parametrizing certain
subvarieties of the compact dual $\check X$ of $X$ (see Sections \ref{subvarX}, \ref{SPBs}, and \ref{mainconst} for the details). Similarly, for a triple $T=(X_\gH,X_1\times X_2)$, as above, we denote by $P_T$ the standard principle bundle associated with a Shimura variety $S_T$ corresponding to $X_\gH$ and, by $\Omega_T$ a Chow subvariety contained in $\Omega$ parametrizing certain subvarieties of $\check{X}_\gH$. As such, we obtain algebraic morphisms
\begin{center}
\begin{tikzcd}
    & P_T\times\Omega_T \arrow[d, "\pi_T"] \arrow[r, "\iota_{P_T}"] & P\times\Omega \arrow[d] \\
S_2 & S_T \arrow[l, "\phi_T"'] \arrow[r]                            & S,                      
\end{tikzcd}
\end{center}
where $S_2$ is a Shimura variety corresponding to $X_2$, and, for a subset $\Pi$ of $P\times\Omega$, we denote by $\Pi_T$ the union of the Zariski closures of the fibers of the map $\phi_T$ restricted to $\pi_T(\iota_{P_T}^{-1}(\Pi))$.

\begin{teo}\label{thm:main}
Let $S$ be a Shimura variety and let $V$ be an irreducible algebraic subvariety of $S$. For every $d\in\NN$, there exists a locally closed algebraic subset $\Pi(d)$ of $P\times\Omega$ and, for each of the irreducible components $\Pi(d)^\circ$ of $\Pi(d)$, an associated triple $T=(X_\gH,X_1\times X_2)$ such that the following holds.

  \begin{itemize}
  \item[(i)] Let $\Pi(d)^\circ$ be an irreducible component of $\Pi(d)$ and let $T$ be its associated triple. Then $\Pi(d)^\circ_T$ is constructible.
  \item[(ii)] Let $W$ be a weakly optimal subvariety of $V$ of weakly special defect $d$. Then there exists an irreducible component
    $\Pi(d)^\circ$ of $\Pi(d)$ (with its associated triple $T$) such that $W$ is the image in $S$ of some irreducible component of a fiber $\phi_T\rest{\Pi(d)^\circ_T}$.

  \item[(iii)] Let $\Pi(d)^\circ$ denote an irreducible component of
    $\Pi(d)$ (with its associated triple $T$). If $W$ is an irreducible component of a fiber of
    $\phi_T\rest{\Pi(d)^\circ_T}$ then the image in $S$ of $W$ is a weakly optimal subvariety of
    $V$ of weakly special defect $d$.
  \end{itemize}
\end{teo}

In short, we show that the weakly optimal subvarieties of $V$ are precisely the
fibers of finitely many constructible families. 

However, our construction allows us to go further, applying estimates from differential
algebraic geometry to obtain effective upper
bounds for the degrees of the families $\Pi(d)$, as well as the
individual weakly optimal subvarieties. Moreover, our construction
also produces an effective description of this {\it unlikely intersection}
locus, in the form of an explicit system of algebraic equations. 

To
formulate such a result explicitly, we fix a
coordinate system on the set $P\times\Omega$ as follows. Let $U\subset S$ denote an open dense subset such that $P$ trivializes
over $U$, i.e. $P\vert_U\simeq U\times\gG(\CC)$. Fixing a faithful
representation of $\gG$ we may realize it as a subvariety of some affine space. We use (1) a very-ample power of the Baily-Borel line bundle for a
set of projective coordinates on $U$, (2) the matrix entries as
coordinates on $\gG(\CC)$, and (3) the standard Chow coordinates on the Chow
variety $\Omega$. All degrees below on $U\times \gG(\CC)\times\Omega$ are
taken with respect to the Segre product of these three coordinate systems.

If $X$ is a locally-closed subset of $U\times \gG(\CC)\times\Omega$
then it is of the form $Y\setminus Z$ where $Y$ and $Z$ are Zariski
closed. We define the {\it complexity} of $X$ to be
$\deg(Y)+\deg(Z)$. We give a bound for the complexity of the sets
$\Pi(d)$, as well as a procedure for computing the equations and
inequations defining $\Pi(d)$, in terms of the equations defining
$V$. The construction depends on certain data associated with $S$,
namely, (1) the equations describing the projective embedding of
$U\times G\times\Omega$ and (2) the equations describing the canonical
connection $\nabla$ on $P$ (written in the chosen coordinate system on
$U\times \gG(\CC)$). Throughout the paper, with the exception of
Section~\ref{sec:fully-effective}, our effective bounds are assumed to
depend on this data (we simply say {\it depending only on $S$}). In
the final Section~\ref{sec:fully-effective}, we address the issue of
obtaining fully effective computations in the case $S=\cA_g$ using
Gauss-Manin connections.

\begin{teo}\label{degbound:intro}
  The complexity of $\Pi(d)$ is bounded by $f_S(\deg(V))$ for some
  polynomial $f_S$ depending only on $S$. Moreover, we produce an
  explicit system of equalities and inequations for $\Pi(d)$ of
  degrees bounded by $f_S(\deg(V))$.
\end{teo}

According to Theorem~\ref{thm:main} the sets $\Pi(d)$ provide a
complete parametrization for all families of weakly optimal
subvarieties of $V$ of weakly special defect $d$. Theorem~\ref{degbound:intro} therefore
provides a method for explicitly computing these families and
controlling their degrees. From this, we obtain degree bounds on individual weakly optimal subvarieties.

\begin{teo}\label{degreebound:intro}
  Let $d\in\NN$ and let $W$ be a weakly optimal subvariety of $V$ of weakly special defect $d$. Then 
  \begin{align*}
\deg(W)\leq f_S(\deg(V))
\end{align*}
for some polynomial $f_S$ depending only on $S$.
\end{teo}

\subsection{Fully effective computation}
\label{sec:intro-fully-effective}

Even though our main construction is effective, to carry out this
procedure in practice one would need to obtain
\begin{itemize}
\item[(1)] an explicit system of equations for the projective embedding of
  $S$ with respect to (some power of the) Baily-Borel line bundle;
\item[(2)] an explicit description of the canonical connection $\nabla$
  with respect to a prescribed trivialization of the corresponding
  bundle $P$;
\item[(3)] an explicit system of equations for the subvariety $V$ with
  respect to the projective embedding.
\end{itemize}

Computing the first two of these items for a given Shimura variety $S$
appears to be a non-trivial task. Moreover, computing equations for
some subvariety $V$ of interest is by itself a non-trivial
problem. For example, consider the case $S=\mathcal{A}_g$, the space of
principally polarized abelian varieties, and $V$, the closure $\mathcal{T}_g$ of the
open Torelli locus
$\mathcal{T}^\circ_g\subset\mathcal{A}_g$. Computing the
weakly special locus is a problem of significant interest
due to its relation to the Coleman--Oort conjecture on the finiteness of the set of isomorphism
classes of genus--$g$ Jacobians with complex multiplication. We refer to \cite{Moonen2013} for an excellent survey and simply recall here that it is conjectured that for $g\gg1$ the set of positive-dimensional special subvarieties of $\mathcal{T}_g$ intersecting $\mathcal{T}^\circ_g$ is empty. In combination with the Andr\'e-Oort conjecture for
$\mathcal{A}_g$ (established by Tsimerman in \cite{tsimerman:AO}), this would imply
the Coleman-Oort conjecture. Note,
however, that computing a set of equations for $V$ in this case is the
famous Schottky problem, a subject of substantial independent study, still not fully resolved.

Fortunately, taking advantage of the functoriality of the canonical
connections on Shimura varieties, we are able to alternatively carry
out the computation for $V\subset S$ described via a moduli
interpretation instead of an explicit projective embedding. We focus
our attention on the case $S=\mathcal{A}_g$. In light of the moduli
interpretation of $\mathcal{A}_g$, it is natural to describe a
subvariety $V\subset\mathcal{A}_g$ as a family of genus--$g$
principally polarized abelian varieties, or the Jacobians of a family
of genus $g$ curves. In this case, the computation of the canonical
connection translates into the computation of the Gauss--Manin
connection for the corresponding family. This is a classical problem
and it is well known, going back to the work of Manin
\cite{manin:kernel}, that it can be carried our explicitly.

We will focus here on the case of families of Jacobians, which lends
itself more readily to effective computation, as curves are relatively
simple to describe using explicit equations. This is also the case
required in principle to treat the Coleman--Oort conjecture and related
problems on the Torelli locus. Our methods could in principle also be
used to carry out explicit computations with more general constructions,
such as Prym varieties, or even with general families of polarized
abelian varieties presented by explicit equations. However, since it
is far less common to present general abelian varieties in this way, we
do not pursue this matter explicitly.

Let $V$ denote an algebraic variety, which we assume for simplicity to
be smooth. Let $T\to V$ denote a smooth curve over $V$, by which we
mean that $T$ is smooth and the map $T\to V$ is submersive.

To fit our general formalism, we should work with a neat subgroup
$\Gamma\subset\gGSp_{2g}(\ZZ)$. We therefore denote by
$f:\tilde V\to V$ an \'etale cover, $\tilde T\to\tilde V$ the
base change of $T$ by $f$, and choose $f$ such that $\tilde T$ is
compatible with an $N$-level structure (say for $N=3$). We denote by
$\iota:\tilde V\to\cA_{g,N}$ the corresponding moduli map.

\begin{teo}\label{thm:effective-computation}
  Given an explicit system of equations for (a projective embedding
  of) the family $T\to V$, one can explicitly construct an affine
  cover $\{V_\alpha\subset V\}$, such that for each $V'=V_\alpha$:
  \begin{enumerate}
  \item $\iota^*P\rest{f^{-1}(V')}\simeq V'\times\gGSp(\CC)$,
  \item for each $d$, one can explicitly construct a system of
    algebraic equations and inequations for sets
    $\Pi'(d)\subset V'\times\gGSp(\CC)$ such that
    $f^*\Pi'(d)=\iota^*\Pi(d)\rest{f^{-1}(V')}$.
  \end{enumerate}
\end{teo}

Theorem~\ref{thm:effective-computation} in principle provides an
effective method for deciding the question of whether weakly special
subvarieties exist within the open Torelli locus
$\mathcal{T}^\circ_g$ for any given $g$. One first has to explicitly
write down equations for a family $T\to V$ parametrizing all genus $g$
curves. One then computes explicitly the equations for the defect-zero
set $\iota^*\Pi(0)$ corresponding to families of weakly special
subvarieties contained in $V$. Finally, one may use effective
commutative algebra methods (for instance, Gr\"obner basis algorithms)
to determine whether the family is empty.

\begin{rem}
  In the particular case of checking for the existence of
  weakly-special subvarieties, the smoothness assumption on $V$ can be
  dropped. One can first effectively construct a smooth open dense
  subset $V'\subset V$ and apply the preceding process to $V'$, and
  then proceed by induction on dimension with $V\setminus V'$.
\end{rem}

One may apply a similar process to look for weakly special or
weakly optimal subvarieties of other families of interest, for
instance, the family of (Jacobians of) hyperelliptic curves of a given
genus $g$.  We stress that, to perform this computation using
Theorem~\ref{degbound:intro}, one would need to obtain an explicit
description of the hyperelliptic locus $V\subset\mathcal{A}_g$, which
is a non-trivial problem.

\subsection{Literature review}
Effective results on the Andr\'e--Oort and the Zilber--Pink conjectures are still relatively sparse. This work improves upon a previous work of the second author with Javanpeykar and K\"uhne \cite{DJK} which, by entirely different methods, gave effective degree bounds for so-called {\it non-facteur} maximal special subvarieties. We refer to the introduction of the latter for references to several earlier works. More recently, the first author has obtained effective results on the Andr\'e--Oort conjecture for products of modular curves \cite{Binyamini:AOYn} and, with Masser, for Hilbert modular varieties \cite{BM:AOHilb}. Very recently, Pila and Scanlon have announced effective results for function field versions of the Zilber-Pink conjectures for varieties supporting a variation of Hodge structures, also using differential algebraic methods \cite{PS:diffZP}.

\subsection{Acknowledgements}

G. Binyamini was supported by the ISRAEL SCIENCE FOUNDATION (grant No. 1167/17) and by funding from the European Research Council (ERC) under the European Union's Horizon 2020 research and innovation programme (grant agreement No 802107). C. Daw would like to thank the EPSRC for its support via a New Investigator Award (EP/S029613/1). He would also like to thank the University of Oxford for having him as a Visiting Research Fellow.

\section{Generalities}

\subsection{Analytic and algebraic sets}
For a complex analytic set $X$ and $x\in X$, we write $\dim_xX$ for
the dimension of $X$ at $x$, as defined in \cite{gr:sheaves}. We write
$\dim X$ for the supremum of $\dim_xX$ for all $x\in X$. We recall
that, if $X$ is irreducible, $\dim_xX$ is constant on $X$. If $X$ is
(explicitly) an algebraic variety, then $\dim X$ will refer to its
dimension as an algebraic variety. All algebraic subvarieties are
assumed to be (Zariski) closed, unless stated otherwise.

\subsection{Degrees}\label{degrees}

If $X$ is a complex algebraic variety and $k\in\NN$, we denote by
$A_kX$ the group of $k$--cycles modulo rational equivalence on $X$
(see \cite[Section 1.3]{Fulton1998}). We define the degree
$\deg(\alpha)$ of $\alpha\in A_kX$ as in \cite[Definition
1.4]{Fulton1998}. In particular, $\deg(\alpha)=0$ if $k>0$. If
$\alpha\in A_kX$ and $L$ is a line bundle on $X$, we obtain, for any
positive integer $d\leq k$, a class
\begin{align*}
  c_1(L)^d\cap\alpha\in A_{k-d}X
\end{align*}
(see \cite[Definition 2.5]{Fulton1998}). If $V$ is an irreducible
subvariety of $X$, we define the degree $\deg_L(V)$ of $V$ with
respect to $L$ to be the degree of
\begin{align*}
  c_1(L)^{\dim V}\cap[V]\in A_0X,
\end{align*}
where $[V]\in A_{\dim V}X$ denotes the class of the $\dim V$--cycle
given by $V$. If $V$ is a (not necessarily irreducible) subvariety of
$X$, we define the $\deg_L(V)$ to be the sum of the $\deg_L(V_i)$, as
$V_i$ varies over the irreducible components of $V$.

\subsection{Algebraic groups}

For an algebraic group $\gG$, we denote by $\gG^\circ$ the connected
component (in the Zariski topology) of $\gG$ containing the identity,
and we denote by the corresponding mathfrak letter $\mathfrak{g}$ its
Lie algebra.

We include connected in our definitions of reductive and semisimple
algebraic groups. For a reductive algebraic group $\gG$, we denote by
$\gG^\ad$ the quotient of $\gG$ by its center $\gZ(\gG)$, and we
denote by $\gG^\der$ the derived subgroup of $\gG$. If $\gG$ is
defined over (a field containing) $\RR$, we denote the connected
component (in the archimedean topology) $\gG(\RR)^+$ of $\gG(\RR)$
containing the identity by the Roman letter $G$, retaining any
superscripts or subscripts, and we write $\gG(\RR)_+$ for the inverse
image of $G^\ad$ under the natural map
$\gG(\RR)\rightarrow\gG^\ad(\RR)$. We write $\gG(\QQ)_+$ for
$\gG(\RR)_+\cap\gG(\QQ)$.

If $\gG$ is a reductive algebraic group over a field of characteristic
zero, and $\gH$ is a reductive algebraic subgroup of $\gG$, then we
write $\gN_\gG(\gH)$ (resp. $\gZ_\gG(\gH)$) for the normalizer
(resp. the centralizer) of $\gH$ in $\gG$. We recall that
$\gN_\gG(\gH)^\circ$ and $\gZ_\gG(\gH)^\circ$ are both reductive. We
have an almost direct product decomposition
$\gN_\gG(\gH)^\circ=\gH\cdot\gZ_\gG(\gH)^\circ$.

\subsection{Arithmetic groups}

Let $\gG$ denote a reductive algebraic group over $\QQ$ and, via a
faithful representation, consider $\gG$ as a subgroup of $\gGL_n$, for
some $n\in\NN$. The definitions that follow are independent of this
representation and, hence, we can and do make use of them without
reference to such a representation.

An arithmetic subgroup of $\gG(\QQ)$ is any subgroup commensurable
with $\gG(\QQ)\cap\gGL_n(\ZZ)$. An element of $\gG(\QQ)$ is neat if
the subgroup of $\bar{\QQ}$ generated by its eigenvalues (considering
it as an element of $\gGL_n(\bar{\QQ})$) is torsion free. A subgroup
of $\gG(\QQ)$ is neat if all of its elements are neat. In particular,
a neat subgroup is torsion free.

\section{Shimura varieties and the Zilber--Pink conjecture}

\subsection{Shimura data}\label{SD}

Let $\SSS$ denote the Deligne torus, that is, the Weil restriction
from $\CC$ to $\RR$ of $\GG_m$. By a Shimura datum, we refer to a pair
$(\gG,\gX)$, where $\gG$ is a reductive algebraic group defined over
$\QQ$ and $\gX$ is a $\gG(\RR)$--conjugacy class of morphisms
$\SSS\rightarrow\gG_\RR$ such that the conditions SV1, SV2, and SV3 of
\cite[p50]{milne:intro} hold. Furthermore, we impose the condition
\begin{itemize}
\item[(SV0)] $\gG$ is the generic Mumford--Tate group on $\gX$.
\end{itemize}

Condition SV0 means that $\gG$ is the smallest algebraic subgroup
$\gH$ of $\gG$ defined over $\QQ$ such that $x(\SSS)$ is contained in
$\gH_\RR$ for all $x\in\gX$. We recall that $\gX$ is naturally a
disjoint union of hermitian symmetric domains. We refer the reader to
\cite{milne:intro} for more details regarding the theory of Shimura
varieties.

In this article, in order to simplify technical issues, we will assume
that our ambient Shimura datum $(\gG,\gX)$ satisfies the condition
that $\gZ(\gG)(\RR)$ is compact.

\subsection{Shimura varieties}\label{SVs}

Let $(\gG,\gX)$ be a Shimura datum such that $\gZ(\gG)(\RR)$ is
compact and let $X$ be a connected component of $\gX$. As in
\cite{DO:ECM}, we refer to the pair $(\gG,X)$ as a Shimura datum
component. Let $\Gamma$ be an arithmetic subgroup of $\gG(\QQ)$
contained $\gG(\QQ)_+$. Then $\Gamma$ acts on $X$ and, by the theorem
of Baily--Borel \cite{bb:compactification}, the quotient
$S=\Gamma\backslash X$ naturally possesses the structure of an
irreducible quasi-projective complex algebraic variety. Indeed, by
\cite[Lemma 10.8]{bb:compactification}, the line bundle of holomorphic
forms of maximal degree on $X$ descends to an ample line bundle
$L_\Gamma$ on $S$. Note that, if $\Gamma$ is neat, then $S$ is
non-singular (see \cite[Facts 2.3]{pink:published}). We let $k_\Gamma$ denote the smallest integer such that $L^{\otimes k_\Gamma}_\Gamma$ is very ample.

We will refer to the irreducible variety $S$ as the Shimura variety
associated with $(\gG,X)$ and $\Gamma$. We will denote by $\pi$ the
natural complex analytic map $X\rightarrow S$.

\subsection{Special subvarieties and the Zilber--Pink conjecture}\label{SSZP}
Recall the situation described in Section \ref{SVs}. Let
$(\gH,\gX_\gH)$ denote a Shimura subdatum of $(\gG,\gX)$ (in
particular, $\gH$ is the generic Mumford--Tate group on $\gX_\gH$),
and let $X_\gH$ denote a connected component of $\gX_\gH$ contained in
$X$. For any arithmetic subgroup $\Gamma_\gH$ of $\gH(\QQ)$ contained
in $\gH(\QQ)_+$, we obtain a Shimura variety
$S_{\gH}=\Gamma_\gH\backslash X_\gH$ and, when $\Gamma_\gH$ is
contained in $\Gamma$, the natural complex analytic map
$\Gamma_\gH\backslash X_\gH\rightarrow\Gamma\backslash X$ is a finite
(hence, closed) morphism of algebraic varieties (see \cite[Facts
2.6]{pink:published}). We refer to the image of any such map as a
special subvariety of $S$.

It is straightforward to show that the intersection of two special
subvarieties is a finite union of special subvarieties. In particular,
for any subvariety $W$ of $S$, there exists a smallest special
subvariety $\langle W\rangle$ of $S$ containing $W$. In light of this,
we define the defect $\delta(W)$ of $W$ by
\begin{align*}
  \delta(W)=\dim\langle W\rangle-\dim W.
\end{align*}

Now fix a subvariety $V$ of $S$. We define an irreducible subvariety
$W$ of $V$ to be optimal in $V$ if, whenever $W\subsetneq Y$ for some
other irreducible subvariety $Y$ of $V$, we have
$\delta(W)<\delta(Y)$. (In particular, $V$ is an optimal subvariety of
$V$ itself.)

Note that an optimal subvariety $W$ of $V$ such that $\delta(W)=0$ is
a maximal special subvariety of $V$. Observe also that an optimal
subvariety $W$ of $V$ is necessarily an irreducible component of
$\langle W\rangle\cap V$.

We denote by $\Opt(V)$ the set of optimal subvarieties of $V$.  The
central problem in the area of {\it unlikely intersections} in Shimura
varieties is (equivalent to) the following (see \cite{DR}).

\begin{conj}[Zilber--Pink]
  Let $V$ be a subvariety of a Shimura variety $S$. Then the set
  $\Opt(V)$ is finite.
\end{conj}

\subsection{Weakly special and weakly optimal
  subvarieties}\label{wswo}

Recall the situation described in Section \ref{SVs}. Let
$(\gH,\gX_\gH)$ denote a Shimura subdatum of $(\gG,\gX)$ and let
$X_\gH$ denote a connected component of $\gX_\gH$ contained in
$X$. Then the image $S_\gH$ of $X_\gH$ in $S$ is a special
subvariety. Decompose $\gH^\ad$ as a product $\gH_1\times\gH_2$ of two
(permissibly trivial) normal $\QQ$--subgroups. In this way, we obtain
a decomposition $X_\gH=X_1\times X_2$, and we will refer to a
decomposition of this form as a $\QQ$--splitting. For any
$x_2\in X_2$, the image $S_{\gH,x_2}$ of $X_1\times\{x_2\}$ in $S$ is
a closed irreducible algebraic subvariety and we refer to any
subvariety of this form as a weakly special subvariety of $S$. We
remark that any special subvariety is weakly special, whereas
$S_{\gH,x_2}$ is special if and only if the Mumford--Tate group of
$x_2$ is a torus (or, equivalently, if $x_2$ is a pre-special point of
$X_2$, to use another terminology).

It is straightforward to show that the intersection of two weakly
special subvarieties is a finite union of weakly special
subvarieties. In particular, for any subvariety $W$ of $S$, there
exists a smallest weakly special subvariety
$\langle W\rangle_{\rm ws}$ of $S$ containing $W$. In light of this,
we define the weakly special defect $\delta_{\rm ws}(W)$ of $W$ by
\begin{align*}
  \delta_{\rm ws}(W)=\dim\langle W\rangle_{\rm ws}-\dim W.
\end{align*}

Now fix a subvariety $V$ of $S$. We define an irreducible subvariety
$W$ of $V$ to be weakly optimal in $V$ if, whenever $W\subsetneq Y$
for some other irreducible subvariety $Y$ of $V$, we have
$\delta_{\rm ws}(W)<\delta_{\rm ws}(Y)$. (In particular, $V$ is a
weakly optimal subvariety of $V$ itself.)

Note that a weakly optimal subvariety $W$ of $V$ such that
$\delta(W)=0$ is a maximal weakly special subvariety of $V$. Observe
also that a weakly optimal subvariety $W$ of $V$ is necessarily an
irreducible component of $\langle W\rangle_{\rm ws}\cap V$. By
\cite[Corollary 4.5]{DR}, any optimal subvariety of $V$ is weakly
optimal.

\begin{rem}
  Note that any point $z\in S$ is a weakly special subvariety (according to our convention, at least). In
  particular, $\delta_{\rm ws}(z)=0$. Hence, $z\in V$ is a weakly
  optimal subvariety of $V$ if and only if it is not contained in a
  positive dimensional weakly special subvariety contained in $V$.
\end{rem}

\section{The (weak) hyperbolic Ax-Schanuel conjecture}

\subsection{Subvarieties of $X$}\label{subvarX}
Recall the situation described in Section \ref{SD} and let $X$ be an
irreducible component of $\gX$. Recall that $X$ is a $G$--conjugacy
class of homomorphisms $\SSS\rightarrow\gG_\RR$. By extending scalars
to $\CC$ and pre-composing with
$\GG_m\rightarrow\GG^2_m\cong\SSS_\CC$, where the first map is given
by $z\mapsto(z,1)$, we obtain from each point $x\in X$ a cocharacter
$\mu_x:\GG_m\rightarrow\gG_\CC$ such that
\begin{itemize}
\item[$(*)$] in the action of $\GG_m$ on $\Lie(\gG_\CC)$, obtained via
  restriction of the adjoint representation, only the characters $z$,
  $1$, and $z^{-1}$ occur.
\end{itemize}

In this way, we obtain an embedding of $X$ into a
$\gG(\CC)$--conjugacy class $X^{\rm co}$ of cocharacters of $\gG_\CC$
satisfying $(*)$. For each $\mu\in X^{\rm co}$ and $r\in\{1,0,-1\}$,
we define $V^r_\mu$ to be the character subspace of $\Lie(\gG_\CC)$ on
which $\GG_m$ acts (according to the action obtained from $\mu$) via
the character $z^r$. Then $\Lie(\gG_\CC)=\oplus_rV^r_\mu$, and we
obtain a filtration $F_\mu$ of $\Lie(\gG_\CC)$ by setting
$F^p_\mu=\oplus_{r\geq p}V^r_\mu$. In this way, we obtain a
$\gG(\CC)$--invariant surjective map $\mu\mapsto F_\mu$ from
$X^{\rm co}$ to a $\gG(\CC)$--orbit of filtrations $\check{X}$. Note
that $\check{X}$ is a complex projective flag variety known as the
compact dual of $X$. The composite map $X\rightarrow\check{X}$ is a
complex analytic $\gG(\RR)_+$--invariant embedding, known as the Borel
embeddding of $X$, and we identify $X$ with its image, which is an
open subset of $\check{X}$. In particular, $\dim\check{X}=\dim X$. We
define a subvariety of $X$ to be any irreducible analytic component of
$X\cap Y$ for any algebraic subvariety $Y$ of $\check{X}$. As noted in
the paragraph following Theorem 5.4 of \cite{DR}, this definition
agrees with the definition therein.

\subsection{Pre-special and pre-weakly special subvarieties}\label{pre-sp}
Recall the situation described in Section \ref{SD} and let $X$ be an
irreducible component of $\gX$.  If $(\gH,\gX_\gH)$ is a Shimura
subdatum of $(\gG,\gX)$ and $X_\gH$ is a connected component of
$\gX_\gH$ contained in $X$, we obtain a commutative diagram of complex
analytic $\gH(\RR)_+$--invariant embeddings
\begin{align*}
  \xymatrix{
  X_\gH \ar[r] \ar[d]  &\check{X}_\gH \ar[d] \\
  X \ar[r]             &\check{X},
                         }
\end{align*}
and we identify all objects with their images in $\check{X}$.

\begin{lem}\label{intwithcompact}
  We have $X\cap\check{X}_\gH=X_\gH$.
\end{lem}

\begin{proof}

  By \cite[VI.B.11]{GGK}, the intersection is contained in
  $X\cap \gX_\gH$. Therefore, let $x_1,x_2\in X\cap \gX_\gH$ and let
  $X_i$ denote the $H$--orbit of $x_i$ in $X$ (in other words, the
  connected component of $\gX_\gH$ containing $x_i$). Let $K_i$ denote
  the maximal compact subgroup of $G$ stabilizing $x_i$ and let
  $G=P_iK_i$ denote the corresponding Cartan decomposition. We also
  have Cartan decompositions $H=(P_i\cap H)(K_i\cap H)$.

  Writing $x_2=\alpha x_1$ for some $\alpha\in G$, we have
  $K_2=\alpha K_1\alpha^{-1}$ and $P_2=\alpha P_1\alpha^{-1}$. Since
  Cartan decompositions are conjugate, we also have
  $K_2\cap H=h(K_1\cap H)h^{-1}$ and $P_2\cap H=h(P_1\cap H)h^{-1}$
  for some $h\in H$.

  We set $\gamma=h^{-1}\alpha$ and write $\gamma=pk$ for some
  $p\in P_1$ and some $k\in K_1$. We deduce that
  $K_1\cap H=pK_1p^{-1}\cap H$ and (trivially)
  $P_1\cap H=pP_1p^{-1}\cap H$. Using \cite[Lemme
  3.11]{ullmo:equidistribution}, as in the proof of \cite[Lemma
  3.7]{uy:andre-oort}, we deduce that $p^2$ centralizes $H$. Since $H$
  is Zariski dense in $\gH_\RR$, it follows that
  $p^2\in\gZ_\gG(\gH)(\RR)$. Therefore, since $p^2\in G$, we conclude
  that $p^2\in K_1$ and so $p^2$ is trivial. It follows that $p$ is
  fixed by the Cartan involution associated with $K_1$ and so
  $p=1$. We conclude that $x_2=hx_1$ and so $x_2\in X_1$, which
  finishes the proof.
\end{proof}

In particular, $X_\gH$ is a subvariety of $X$, and we call such a
subvariety a pre-special subvariety. A similar discussion shows that,
for a $\QQ$--splitting $X_\gH=X_1\times X_2$ as above and a point
$x_2\in X_2$, the set $X_{\gH,x_2}=X_1\times\{x_2\}$ is again a
subvariety of $X$ (we define $\check{X}_{\gH,x_2}$ analogously), and
we refer to such a subvariety as a pre-weakly special subvariety of
$X$.

\begin{rem}\label{Dbound}
  It is an easy consequence of \cite[Lemma 3.7]{uy:andre-oort} and the
  fact that $\gG_\RR$ possesses only finitely many
  $\gG(\RR)$--conjugacy classes of semisimple subgroups that the
  pre-weakly special subvarieties of $X$ belong to finitely many
  $\gG(\RR)$--orbits. It follows that, for a given embedding of
  $\check{X}$ into projective space, there exists a $D\in\NN$ such
  that, for any pre-weakly special subvariety $X_{\gH,x_2}$ of $X$,
  the degree of $\check{X}_{\gH,x_2}$ is at most $D$.
\end{rem}

\subsection{Intersection components}
Recall the situation described in Section \ref{SVs} and let $V$ be an
irreducible subvariety of $S$. We define an intersection component of
$\pi^{-1}(V)$ to be an irreducible analytic component of the
intersection of $\pi^{-1}(V)$ with a subvariety of $X$. For any
intersection component $A$ of $\pi^{-1}(V)$, there exists a smallest
subvariety $\langle A\rangle_{\rm Zar}$ of $X$ containing $A$ (from
which it follows that $A$ is automatically an irreducible analytic
component of $\langle A\rangle_{\rm Zar}\cap\pi^{-1}(V)$). In light of
this, we define the Zariski defect $\delta_{\rm Zar}(A)$ of $A$ by
\begin{align*}
  \delta_{\rm Zar}(A)=\dim\langle A\rangle_{\rm Zar}-\dim A.
\end{align*}
We say that $A$ is Zariski optimal in $\pi^{-1}(V)$ if, whenever
$A\subsetneq B$ for some other intersection component $B$ of
$\pi^{-1}(V)$, we have $\delta_{\rm Zar}(A)<\delta_{\rm Zar}(B)$. The
weak hyperbolic Ax--Schanuel conjecture, which follows (see
\cite[Lemma 5.16]{DR}) from the hyperbolic Ax--Schanuel , proven by
Mok--Pila--Tsimerman \cite{MPT:AS}, is the following.

\begin{teo}[weak hyperbolic Ax--Schanuel]\label{WAS}
  Let $A$ be a Zariski optimal intersection component of
  $\pi^{-1}(V)$. Then $\langle A\rangle_{\rm Zar}$ is pre-weakly
  special.
\end{teo}

\section{Standard principal bundles and canonical foliations}

\subsection{Standard principal bundles}\label{SPBs}

Recall the situation described in Section \ref{SVs} and suppose that
$\Gamma$ is neat. Since $\gZ(\gG)(\RR)$ is compact, the stabilizer in
$\gG(\RR)$ of any point in $X$ is compact. Therefore, since $\Gamma$
is torsion free, it acts without fixed points on $X$. It follows that
$\Gamma$ is (isomorphic to) the fundamental group $\pi_1(S)$ of $S$
and that
\begin{align*}
  P=\Gamma\backslash (X\times\gG(\CC))
\end{align*}
is a principal complex analytic $\gG(\CC)$--bundle over $S$. (The
action of $\Gamma$ on $X\times\gG(\CC)$ is diagonal and on the left,
and the action of $\gG(\CC)$ is given by $h\cdot [x,g]=[x,gh^{-1}]$,
where we use $[x,g]$ to denote the class of $(x,g)\in X\times\gG(\CC)$
in $P$.) Furthermore, there is a canonical flat connection $\nabla$ on
$P$.

Following convention, we refer to $P=(P,\nabla)$ as the standard
principal bundle associated with $(\gG,X)$ and $\Gamma$. By
\cite[Proposition 3.2]{milne:connected}, $P$ is complex algebraic as a
bundle over the algebraic variety $S$. We let $\pi_P:P\rightarrow S$
denote the natural (complex algebraic) morphism $[x,g]\mapsto\pi(x)$.

Note that there is also a natural $\gG(\CC)$--equivariant algebraic
map $\beta:P\rightarrow \check{X}$ defined by
\begin{align*}
  \beta([x,g])=g^{-1}F_{\mu_x}.
\end{align*}
We observe that the composite of $\beta$ with the natural embedding
$X\rightarrow P$ given by $x\mapsto [x,1]$ yields the Borel embedding
of $X$ into $\check{X}$.

\subsection{Trivializations}
Recall the situation described in Section \ref{SPBs}. Let
$p=[x,g]\in P$ and let $U\subset X$ be an open neighbourhood of $x$
such that
\begin{align*}
  \gamma\in\Gamma\text{ and }U\cap\gamma U\neq\emptyset\implies\gamma=1.
\end{align*}

Such a $U$ exists by \cite[Proposition 2.5]{milne:MF} and the fact
that $\Gamma$ is torsion free. It follows immediately that
$\pi_U:U\rightarrow\pi(U)$ is biholomorphic and we obtain a (complex
analytic) trivialization
\begin{align*}
  \varphi_U:\pi(U)\times\gG(\CC)\rightarrow \pi_P^{-1}(\pi(U))
\end{align*}
of $P$ over $\pi(U)$ defined by $(s,g)\mapsto [\pi_U^{-1}(s),g]$.

\subsection{Flat structures}
Recall again the situation described in Section \ref{SPBs}. Choose an
open covering $\mathcal{C}$ of $X$, stable under translation by
$\Gamma$ and such that each point $x\in X$ is contained in an
arbitrarily small $U\in\mathcal{C}$. We claim that we can choose
$\mathcal{C}$ such that
\begin{align}\label{cov}
  U_1,U_2\in\mathcal{C}\text{ and }U_1\cap U_2\neq\emptyset\implies U_1\cap\gamma U_2=\emptyset\text{ for all }\gamma\in\Gamma\setminus\{1\}.
\end{align}
To see this, equip $X$ with its usual metric
$d:X\times X\rightarrow\RR$, and consider
$f:X\times\gG(\RR)_+\rightarrow X\times X$ defined by
$(x,g)\mapsto(x,gx)$. Recall that both $d$ and $f$ are proper maps. In
particular, for a compact interval $I\subset\RR$, the preimage
$(d\circ f)^{-1}(I)$ is a compact subset of $X\times\gG(\RR)_+$ and,
therefore, so too is the projection $\Theta$ to $\gG(\RR)_+$. Since
$\Gamma$ is a discrete subgroup of $\gG(\RR)_+$, we conclude that
$\Gamma\cap\Theta$ is finite. Therefore, since $\Gamma$ acts freely on
$X$, there exists $C>0$ such that
\begin{align*}
  \gamma\neq 1\implies d(x,\gamma x)>C,\text{ for all }x\in X.
\end{align*}
Therefore, in order to define $\mathcal{C}$, choose around each point
$x\in X$ a system of arbitraily small relatively compact open
neighbourhoods $U$ such that
\begin{align*}
  x,y\in U\implies d(x,y)<C/2,
\end{align*}
and such that $\mathcal{C}$ is stable under translation by
$\Gamma$. Now let $U_1,U_2\in\mathcal{C}$ and suppose that
$x_1\in U_1\cap U_2$. Furthermore, suppose that
$x_2\in U_1\cap\gamma U_2$ for some
$\gamma\in\Gamma\setminus\{1\}$. We conclude that
\begin{align*}
  C<d(x_1,\gamma x_1)\leq d(x_1,x_2)+d(x_2,\gamma x_1)<C/2+C/2=C,
\end{align*}
which is a contradiction, yielding the claim.

\

Note that condition (\ref{cov}) with $U_2=U_1=U$ implies that we have
trivializations
\begin{align*}
  \varphi_U:\pi(U)\times\gG(\CC)\rightarrow \pi_P^{-1}(\pi(U))
\end{align*}
as before, for any $U\in\mathcal{C}$. Now suppose that
$\pi(U_1)\cap\pi(U_2)\neq\emptyset$ for $U_1,U_2\in\mathcal{C}$. We
obtain a transition map
\begin{align*}
  \varphi_{U_2}^{-1}\circ\varphi_{U_1}:\pi(U_1)\cap\pi(U_2)\times\gG(\CC)\rightarrow \pi(U_1)\cap\pi(U_2)\times\gG(\CC),
\end{align*}
which sends $(s,g)$ to $(s,\gamma^{-1}_sg)$, where $\gamma_s\in\Gamma$
is the unique element such that
\begin{align*}
  \pi^{-1}_{U_1}(s)=\gamma_s\pi^{-1}_{U_2}(s)\in U_1\cap\gamma_s U_2.
\end{align*}
However, by (\ref{cov}), $\gamma=\gamma_s$ is constant on
$\pi(U_1)\cap\pi(U_2)$.

We refer to a covering $\mathcal{C}$ of $X$ satisfying the properties
above as a flat structure for $P$. In particular, a flat structure
$\mathcal{C}$ comes with an associated set
$\{\varphi_U\}_{U\in\mathcal{C}}$ of trivializations. Note that, if
$\mathcal{C}$ and $\mathcal{C}'$ are both flat structures for $P$,
then $\mathcal{C}\cap\mathcal{C}'$ (whose members are precisely those
of the form $U\cap U'$ for $U\in\mathcal{C}$ and $U'\in\mathcal{C}'$)
is also a flat structure for $P$.

\subsection{The canonical foliation}\label{fol}
Recall again the situation described in Section \ref{SPBs}.  By
\cite[Section 1B]{blaine:fol}, we obtain a canonical foliation
$\mathcal{F}$ of $P$. For any $p\in P$, we can obtain the leaf
$\mathcal{L}_p$ of $\mathcal{F}$ through $p$ as follows. Let
$\mathcal{C}$ be a flat structure for $P$, write $p=[x,g_p]$, and let
$U\in\mathcal{C}$ be such that $x\in U$. Then
\begin{align*}
  \varphi^{-1}_U(\mathcal{L}_p\cap\pi^{-1}_P(\pi(U)))=\pi(U)\times\{g_p\}.
\end{align*}
In other words, $\mathcal{L}_p$ is given locally by the images of the
horizontal sections.

\subsection{Intersection dimensions}\label{intdims}

Recall the situation described in Section \ref{fol} and, for a point
$p\in P$, let $\mathcal{L}_p$ denote the leaf of $\mathcal{F}$ through
$p$. Let $V$ denote an irreducible subvariety of $S$. We will make
repeated use of the following lemma.

\begin{lem}\label{intlem}
  Let $p\in P$ and let $Y$ denote a subvariety of $\check{X}$ such
  that $p\in\pi^{-1}_P(V)\cap\beta^{-1}(Y)$. For any choice of
  representation $p=[x,g_p]$, we have
  \begin{align*}
    \dim_p(\cL_p\cap\beta^{-1}(Y)\cap\pi^{-1}_P(V))=\dim_x(g_pY\cap\pi^{-1}(V)).
  \end{align*}

\end{lem}

\begin{proof}
  Write $p=[x,g_p]$. By definition, $x\in g_p Y\cap\pi^{-1}(V)$. Fix a
  flat structure $\mathcal{C}$ for $P$ and consider $U\in\mathcal{C}$
  such that $x\in U$. From Section \ref{fol}, we have
  \begin{align}\label{leafeq}
    \varphi^{-1}_U(\mathcal{L}_p\cap \pi^{-1}_P(\pi(U)))=\pi(U)\times\{g_p\}
  \end{align}
  and, from the definitions, we have
  \begin{align}\label{vareq}
    \varphi^{-1}_U(\pi^{-1}_P(V)\cap \pi^{-1}_P(\pi(U)))=(\pi(U)\cap V)\times\gG(\CC).
  \end{align}
  Finally, we claim that
  \begin{align}\label{geoeq}
    \varphi^{-1}_U(\beta^{-1}(Y)\cap \pi^{-1}_P(\pi(U)))=\{(\pi(gy),g):y\in\check{Y},\ g\in\gG(\CC),\ gy\in U\}.
  \end{align}

  To see (\ref{geoeq}), first note that,
  \begin{align*}
    \beta^{-1}(Y)=\{[gy,g]:y\in Y,\ g\in\gG(\CC)\}.
  \end{align*}
  In particular,
  \begin{align*}
    \beta^{-1}(Y)\cap \pi^{-1}_P(\pi(U))=\{ [gy,g]:y\in Y,\ g\in\gG(\CC),\ \pi(gy)\in\pi(U)\}.
  \end{align*}
  Now, if $\pi(gy)\in\pi(U)$, then $gy\in\gamma U$, for some
  $\gamma\in\Gamma$. That is, $\gamma^{-1}gy=\pi^{-1}_U(\pi(gy))$ and
  so
  \begin{align*}
    \varphi^{-1}_U([gy,g])=(\pi(gy),\gamma^{-1}g)=(\pi(\gamma^{-1}gy),\gamma^{-1}g),
  \end{align*}
  which is an element belonging to the set on the right hand side of
  (\ref{geoeq}). On the other hand, if $y\in Y$, $g\in\gG(\CC)$, and
  $gy\in U$, then $[gy,g]\in \beta^{-1}(Y)\cap \pi^{-1}_P(\pi(U))$
  maps to $(\pi(gy),g)$. This establishes (\ref{geoeq}).

  Combining (\ref{leafeq}), (\ref{vareq}), and (\ref{geoeq}), we
  conclude that
  $\varphi^{-1}_U(\mathcal{L}_p\cap\beta^{-1}(gY)\cap\pi^{-1}_P(V)\cap \pi^{-1}_P(\pi(U)))$
  is equal to the set of tuples $(\pi(g_py),g_p)$, where $y\in Y$ is
  such that $g_py\in U\cap\pi^{-1}(V)$. Therefore, applying
  $\pi^{-1}_U$ to the first factor, we obtain
  \begin{align}\label{inteq}
    \varphi^{-1}_U(\mathcal{L}_p\cap\beta^{-1}(Y)\cap\pi^{-1}_P(V)\cap \pi^{-1}_P(\pi(U)))\cong g_p Y\cap U\cap\pi^{-1}(V),
  \end{align}
  which proves the result.
\end{proof}

\section{The main construction}\label{mainconst}

Recall the situation described in Section \ref{intdims}:
\begin{itemize}
\item[] $(\gG,X)$ is a Shimura datum component such that
  $\gZ(\gG)(\RR)$ is compact;
\item[] $\Gamma$ is a neat arithmetic subgroup of $\gG(\QQ)$ contained
  in $\gG(\QQ)_+$ equal to $\pi_1(S)$, where
\item[] $S$ is the Shimura variety associated with $(\gG,X)$ and
  $\Gamma$;
\item[] $V$ is an irreducible subvariety of $S$;
\item[] $P$ is the standard principal bundle associated with $(\gG,X)$
  and $\Gamma$;
\item[] $\pi:X\rightarrow S$ and $\pi_P:P\rightarrow S$ denote the
  natural maps;
\item[] $\beta:P\rightarrow\check{X}$ denotes the map defined in
  Section \ref{SPBs};
\item[] $\mathcal{F}$ denotes the canonical foliation of $P$ (see
  Section \ref{fol});
\item[] $\mathcal{L}_p$ denotes the leaf of $\mathcal{F}$ through
  $p\in P$.
\end{itemize}

Fix an embedding of $\check{X}$ into projective space and let
$D\in\NN$ be as in Remark \ref{Dbound}. Let
$\Omega(k)=\Omega(\check{X},k,D^{k+1})$ denote the (quasi-projective)
Chow variety of closed irreducible complex subvarieties of $\check{X}$
of codimension at most $k$ and degree at most $D^{k+1}$. Let
\begin{align*}
  \Omega=\cup_{k=0}^{\dim\check{X}}\Omega(k).
\end{align*}
Consider the algebraic subvariety $\Theta=\Theta(V)$ of
$P\times\Omega$ consisting of the tuples $(p,Y)$ such that
$p\in\pi^{-1}_P(V)\cap\beta^{-1}(Y)$. Here we slightly abuse notation by using $Y$ to denote the Chow
  coordinate representing an irreducible variety as well as the
  variety itself. However, the correspondence between the
  Chow coordinate of $Y$ and the points of $Y$ is of course algebraic,
  and $\Theta$ is indeed Zariski closed. Note that there is a natural
morphism $\Theta\to S$ given by the projection to $P$ composed with
$\pi_P$.

We define a function
\begin{align*}
  d=d(V):\Theta\rightarrow\NN\cup\{0\}
\end{align*}
by setting $d(p,Y)=\dim_p(\cL_p\cap\beta^{-1}(Y)\cap\pi^{-1}_P(V))$.

For any $(p,Y_1)\in \Theta$, we let $\delta_1(p,Y_1)$ be the statement
that
\begin{align*}
  \dim Y_1-d(p,Y_1)<\dim Y_2-d(p,Y_2)
\end{align*}
for all $Y_2\in\Omega$ such that $Y_1\subsetneq Y_2$. Similarly, we
let $\delta_2(p,Y_1)$ be the statement that
\begin{align*}
  d(p,Y_2)<d(p,Y_1)
\end{align*}
for all $Y_2\in\Omega$ such that $\beta(p)\in Y_2\subsetneq Y_1$.

We define $\Pi=\Pi(V)$ to be the set of tuples $(p,Y)\in \Theta$ for
which $\delta_1(p,Y)$ and $\delta_2(p,Y)$ hold. For any $d\in\NN$, we
let $\Pi(d)$ denote the set of tuples $(p,Y)\in\Pi$ such that
\begin{align}\label{eq:Pi-d-cond}
  \dim Y-d(p,Y)=d.
\end{align}

Recall the notion of complexity of locally-closed sets introduced
before Theorem~\ref{degbound:intro}. The proof of the following
results relies on differential-algebraic tools developed in
Section~\ref{sec:diff-algebra}, and is presented below to avoid
breaking the logical flow of the paper. However the reader may easily
verify that the contents of Section~\ref{sec:diff-algebra} are self
contained and do not rely on Proposition~\ref{loc-closed}.

\begin{prop}\label{loc-closed}
  The sets $\Pi(d)$ for $d\in\NN$ are locally closed subsets of
  $\Theta$. The complexity of $\Pi(d)$ is bounded $f_S(\deg(V))$ for
  some polynomial $f_S$ depending only on $S$. Moreover, one can
  derive an explicit system of equalities and inequations for
  $\Pi(d)$, as described in Theorem~\ref{degbound:intro}.

  The (algebraic, left) action of $\gG(\CC)$ on
  $P\times\Omega$ defined by $g(p,Y)=(g\cdot p,gY)$ preserves the $\Pi(d)$ and
  their irreducible components.
\end{prop}

\begin{proof}
  Consider the foliation $\cF_0$ on $P\times\Omega$ given by
  the direct product of (1) the canonical foliation of $P$ and (2) the trivial foliation
  by zero-dimensional leafs on $\Omega$.

  Applying Proposition~\ref{prop:Sigma-bound} to the sets
  $\Sigma(\Theta,\cF_0,k)$ we conclude that the sets
  \begin{equation}
    A(k) := \{(p,Y)\in\Theta : d(p,Y)\ge k \}
  \end{equation}
  are Zariski closed with degrees bounded by a polynomial as claimed,
  and that it is possible to effectively compute equations for these
  sets. (This is our principal ingredient from differential algebra,
  and we will apply this below to deduce the same result for the sets
  $\Pi(d)$.)
  
  Let $d\in\N$. Since $\Omega$ is a disjoint union of Chow varieties
  of different dimensions and degrees, it will be enough to consider
  each of these components separately. We now restrict to $\Omega'$
  given by one of these Chow varieties, and assume that
  $\dim Y$ and $\deg(Y)$ are fixed.

  The set $\Delta(d)\subset P\times\Omega'$ defined by
  condition~\eqref{eq:Pi-d-cond} is given by
  $A(\dim Y-d)\setminus A(\dim Y-d+1)$. It is therefore locally
  closed. We claim further that the condition $\delta_1(p,Y)$ is open
  in $\Delta(d)$. To see this, let $\bar\Omega$ denote the projective
  closure of $\Omega$ and consider
  \begin{align*}
    \Delta_1 := \{(p,Y_1,Y_2)\in P\times\Omega'\times\bar\Omega :\ 
    Y_1\subsetneq^* Y_2,\ d \ge\dim Y_2-d(p,Y_2) \}.
  \end{align*}
  Here we write $Y_1\subsetneq^* Y_2$ to mean that $Y_1$ is strictly
  contained in \emph{each} component of the support of $Y_2$. This is
  a Zariski closed condition, and $\Delta_1$ is therefore closed for
  the same reason that $A(k)$ is closed (with similar degree bounds,
  etc.). Since $\bar\Omega$ is projective, the projection
  $\pi_\Theta(\Delta_1)$ is closed as well and the standard resultant
  constructions from elimination theory produce effective systems of
  equations for this set as well.

  By definition, in $\Delta(d)$ the condition $\delta_1(p,Y)$ essentially
  agrees with the complement of $\pi_\Theta(\Delta_1)$, except for a
  minor technicality: in $\delta_1(p,Y)$ we quantify over $Y_2$ in the
  open Chow variety $\Omega$, whereas, in the definition of $\Delta_1$ we
  used the closed $\bar\Omega$. It is, however, easy to see that
  quantifying over $\bar\Omega$ gives an equivalent condition. Indeed, the points of the closed Chow variety
  $\bar\Omega$ represent effective cycles. If there exists
  $Y_2\in\bar\Omega$ with $Y_1\subsetneq^* Y_2$ such that
  \begin{equation}
    \dim Y_1-d(p,Y_1)<\dim Y_2-d(p,Y_2)
  \end{equation}
  then the same must be true for one of the irreducible components of
  the support of $Y_2$ (note that here it is crucial that we used the
  refined relation $\subsetneq^*$). To conclude, in $\Delta(d)$ the
  condition $\delta_1(p,Y)$ is given by the complement of
  $\pi_\Theta(\Delta_1)$, and is therefore locally closed with the
  stated degree bounds and explicit equations.

  In an entirely analogous way, one checks that $\delta_2(p,Y)$ is open
  in $\Delta$, and this concludes the proof of the local-closedness, as
  well as the degree bounds for $\Pi(d)$.

  The fact that $\gG(\CC)$ preserves the
  $\Pi(d)$ is immediate from the
  definitions. Considering the action as a morphism
  $\gG(\CC)\times P\times\Omega\to P\times\Omega$ yields the remaining
  claims.
\end{proof}

\begin{lem}\label{isweaksp}
  Let $d\in\NN$, let $(p,Y)\in\Pi(d)$, and write $p=[x,g_p]$. Then
  \begin{itemize}
  \item[(i)] $x\in g_pY\cap\pi^{-1}(V)$;
  \item[(ii)] if $A$ denotes an irreducible analytic component of
    $g_pY\cap\pi^{-1}(V)$ passing through $x$ such that
    \begin{align*}
      \dim A=\dim_x(g_pY\cap\pi^{-1}(V)),
    \end{align*}
    $A$ is a Zariski optimal intersection component of $\pi^{-1}(V)$;
  \item[(iii)] writing $X_A$ for the pre-weakly special subvariety
    $\langle A\rangle_{\rm Zar}$ of $X$ (see Theorem \ref{WAS}), we
    have $g_pY=\check{X}_A$;
  \item[(iv)] $\delta_{\rm Zar}(A)=d$.
  \end{itemize}
\end{lem}

\begin{proof}
  We will imitate the proof of \cite[Lemma 6.14]{DR}. The fact that
  $x\in g_pY\cap\pi^{-1}(V)$ follows from the definition of
  $\Pi$. Note also that
  \begin{align*}
    \dim_x(g_pY\cap\pi^{-1}(V))=d(p,Y)
  \end{align*}
  by (\ref{inteq}). Therefore, let $A$ be as in (ii). By definition,
  $A$ is an intersection component and $\langle A\rangle_{\rm Zar}$ is
  contained in $g_pY$. Therefore, since $(p,Y)\in\Pi(d)$, we have
  \begin{align*}
    \delta_{\rm Zar}(A)\leq\dim g_pY-\dim A=\dim Y-d(p,Y)= d.
  \end{align*} 

  Let $B$ be an intersection component of $\pi^{-1}(V)$ containing $A$
  such that $\delta_{\rm Zar}(B)\leq\delta_{\rm Zar}(A)$. We can and
  do assume that $B$ is Zariski optimal and, therefore, by Theorem
  \ref{WAS}, $\langle B\rangle_{\rm Zar}$ is equal to a pre-weakly
  special subvariety $X_B$ of $X$.

  Let $Z$ be an irreducible component of $g_pY\cap\check{X}_B$
  containing $A$. Observe that, either $Z=g_pY$, $Z=\check{X}_B$, or
  $\codim Z>\codim g_pY$. In all cases, the degree of $Z$ in
  $\check{X}$ is at most $D^{\codim Z+1}$ and so
  $Z\in\Omega$. However, $\beta(p)\in g^{-1}_pZ\subseteq Y$, and
  \begin{align*}
    d(p,g^{-1}_pZ)=\dim_x(Z\cap\pi^{-1}(V))=\dim A=d(p,Y).
  \end{align*}
  Therefore, since $\delta_2(p,Y)$ holds, we conclude that
  $Z=g_pY$. In particular, $g_pY$ is contained in $\check{X}_B$.

  On the other hand,
  \begin{align*}
    \dim\check{X}_B-d(p,g^{-1}_p\check{X}_B)&=\dim\check{X}_B-\dim_x(\check{X}_B\cap\pi^{-1}(V))\\
                                            &\leq\delta_{\rm Zar}(B)\\
                                            &\leq\delta_{\rm Zar}(A)\\
                                            &\leq\dim Y-d(p,Y)
  \end{align*}
  and, since $\delta_1(p,Y)$ holds, it follows that
  $g_pY=\check{X}_B$. Therefore, $B=A$ and
  \begin{align*}
    \delta_{\rm Zar}(B)=\dim Y-d(p,Y)=d.
  \end{align*}
\end{proof}

\begin{lem}\label{findarep}
  Let $A$ be a Zariski optimal intersection component of $\pi^{-1}(V)$
  and let $d=\delta_{\rm Zar}(A)$. Let $X_A$ denote the pre-weakly
  special subvariety $\langle A\rangle_{\rm Zar}$ of $X$ (see Theorem
  \ref{WAS}). Then, for any $p=[x,1]$ with $x\in A$ satisfying
  \begin{align*}
    \dim_x(\check{X}_A\cap\pi^{-1}(V))=\dim A,
  \end{align*}
  we have $(p,\check{X}_A)\in\Pi(d)$.
\end{lem}

The proof of Lemma \ref{findarep} is very similar to the proof of
\cite[Lemma 6.13]{DR}. However, note that there is a typographical
error in the statement of \cite[Lemma 6.13]{DR}: the term ``pre-weakly
special'' should be replaced by ``Zariski optimal''. This also occurs
in \cite[Proposition 6.10]{DR} and its proof.

\begin{proof}[Proof of Lemma \ref{findarep}]

  Let $p=[x,1]$ with $x\in A$ satisfying
  \begin{align*}
    \dim_x(\check{X}_A\cap\pi^{-1}(V))=\dim A.
  \end{align*}
  Since, $X_A$ is pre-weakly special, $(p,\check{X}_A)\in \Theta$ and
  we will now show that $(p,\check{X}_A)\in\Pi$.

  To that end, suppose that $\delta_1(p,\check{X}_A)$ does not
  hold. Therefore, there exists $Y\in\Omega$ such that
  $\check{X}_A\subsetneq Y$ and
  \begin{align}\label{delta2}
    \dim \check{X}_A-d(p,\check{X}_A)\geq\dim Y-d(p,Y).
  \end{align}
  Recall that
  $d(p,\check{X}_A)=\dim_x(\check{X}_A\cap\pi^{-1}(V))=\dim A$. Let
  $B$ be an irreducible analytic component of $Y\cap\pi^{-1}(V)$
  passing through $x$ such that $\dim B=\dim_x(Y\cap\pi^{-1}(V))$. In
  other words,
  \begin{align*}
    d(p,Y)=\dim_p(\cL_p\cap\beta^{-1}(Y)\cap\pi^{-1}_P(V))=\dim_x(Y\cap\pi^{-1}(V))=\dim B,
  \end{align*}
  and we obtain $\delta_{\rm Zar}(B)\leq\delta_{\rm Zar}(A)$. It
  follows, as in the proof of \cite[Lemma 6.13]{DR}, that
  $B=A$. However, this contradicts (\ref{delta2}), and so
  $\delta_1(p,\check{X}_A)$ holds.

  Now suppose that $\delta_2(p,\check{X}_A)$ does not hold. Therefore,
  there exists $Y\in\Omega$ such that $Y\subsetneq \check{X}_A$ and
  $d(p,Y)=d(p,\check{X}_A)$. However, this implies that
  \begin{align*}
    \dim_x(Y\cap\pi^{-1}(V))=\dim_x(\check{X}_A\cap\pi^{-1}(V)),
  \end{align*}
  from which it follows that $A$ is contained in
  $Y\subsetneq \check{X}_A$. However, this contradicts the fact that
  $X_A=\langle A\rangle_{\rm Zar}$, and so $\delta_2(p,\check{X}_A)$
  holds.

  Finally, since $d(p,\check{X}_A)=\dim A$, we have
  \begin{align*}
    \dim \check{X}_A-d(p,\check{X}_A)=\dim X_A-\dim A=d,
  \end{align*}
  and so $(p,\check{X}_A)\in\Pi(d)$.
\end{proof}

\begin{lem}\label{fibers}
  Let $\Pi^\circ$ denote an irreducible component of $\Pi$. There exists a pre-special subvariety $X_\gH$ of
  $X$ and a $\QQ$--splitting $X_\gH=X_1\times X_2$ such that, for any
  $(p,Y)\in\Pi^\circ$, if we write $p=[x,g]$, then
  $gY=\gamma\check{X}_{\gH,x_2}$ for some $\gamma\in\Gamma$ and some
  $x_2\in X_2$.
\end{lem}

\begin{proof}
  Fix a flat structure $\mathcal{C}$ for $P$ and let
  $(p_0,Y_0)\in\Pi^\circ$. Write $p_0=[x_0,g_0]$ and let
  $U_0\in\mathcal{C}$ be such that $x_0\in U_0$. We have a
  biholomorpic map
  \begin{align*}
    \tilde{U}_0=\pi^{-1}_P(\pi(U_0))\times\Omega\xrightarrow{f_{U_0}}\pi(U_0)\times\gG(\CC)\times\Omega
  \end{align*}
  given by $(p,Y)\mapsto (\varphi^{-1}_{U_0}(p),Y)$. We write
  $\Pi^\circ_{U_0}=\Pi^\circ\cap\tilde{U}_0$.

  Observe that, for any pre-special subvariety $X_{\gH}$ of $X$ and
  any $\QQ$--splitting $X_1\times X_2$ of $X_{\gH}$, the set
  $\cR(X_\gH,X_1,X_2)$ of points
  $(s,g,Y)\in S\times\gG(\CC)\times\Omega$ such that
  $gY=\check{X}_{\gH,x_2}$ for some $x_2\in\check{X}_2$ is a closed
  algebraic subvariety. To see this, let $\cS(X_\gH,X_1,X_2)$ denote
  the closed algebraic subvariety of tuples
  \begin{align*}
    (y,g,Y,x_2)\in\check{X}\times\gG(\CC)\times\Omega\times\check{X}_2
  \end{align*}
  such that $y\in gY\cap \check{X}_{\gH,x_2}$ and let $f$ denote the
  natural projection from $\cS(X_\gH,X_1,X_2)$ to
  $\gG(\CC)\times\Omega\times\check{X}_2$. Observe that the set of
  points $(g,Y,x_2)\in\gG(\CC)\times\Omega\times\check{X}_2$
  satisfying
  \begin{align*}
    \dim f^{-1}((g,Y,x_2))\geq\max\{\dim Y,\dim X_1\}
  \end{align*}
  constitutes a closed algebraic subvariety (to simplify the
  exposition, one may assume that $\dim Y$ is constant on $\Omega$).
  Now the observation follows from the fact that, because
  $\check{X}_2$ is projective, the natural projection from
  $\gG(\CC)\times\Omega\times\check{X}_2$ to $\gG(\CC)\times\Omega$ is
  closed.

  By Lemma \ref{isweaksp}, $f_{U_0}(\Pi^\circ_{U_0})$ is
  contained in the union of the $\mathcal{R}(X_\gH,X_1, X_2)$ as
  $X_\gH$ varies over the pre-special subvarieties of $X$ and
  $X_1\times X_2$ varies over the $\QQ$--splittings of
  $X_\gH$. Furthermore, after possibly replacing $U_0$ with a subset
  (also belonging to $\mathcal{C}$), we may assume that
  $f_{U_0}(\Pi^\circ_{U_0})$ is connected. Therefore, since
  their union is countable, $f_{U_0}(\Pi^\circ_{U_0})$ is
  contained in one of the $\mathcal{R}(X_\gH,X_1, X_2)$, which we
  denote $\mathcal{R}$. 

  By definition, $P\times\Omega$ is covered by the union of the
  $\tilde{U}$ as $U$ varies over the elements of
  $\mathcal{C}$. Therefore, since $\Pi^\circ$ is path-connected and
  the transition functions associated with the trivializations of $P$
  are given by elements of $\Gamma$, we conclude that there exists a
  pre-special subvariety $X_\gH$ of $X$ and a $\QQ$--splitting
  $X_\gH=X_1\times X_2$ such that, for any $(p,Y)\in\Pi$, if we write
  $p=[x,g]$, then $gY=\gamma\check{X}_{\gH,x_2}$ for some
  $\gamma\in\Gamma$ and some $x_2\in \check{X}_2$. To conclude the
  proof, we recall that $x\in gY$. Hence,
  $x\in X\cap\gamma\check{X}_{\gH,x_2}$ and so $x_2\in X_2$, by Lemma
  \ref{intwithcompact}.
\end{proof}

With each $d\in\NN$ and each irreducible component $\Pi(d)^\circ$ of $\Pi(d)$, we associate a triple
$T=(X_\gH,X_1,X_2)$, where $X_\gH$ is a pre-special subvariety $X_\gH$
of $X$ and $X_\gH=X_1\times X_2$ is a $\QQ$--splitting, such that
Lemma \ref{fibers} holds for all $(p,Y)\in\Pi(d)^\circ$. With the
triple $T$ we associate the standard principal bundle
$P_T=\Gamma_\gH\backslash(X_\gH\times\gH(\CC))$ associated with
$(\gH,X_\gH)$ and $\Gamma_\gH=\Gamma\cap\gH(\QQ)_+$. (This is indeed a
principal bundle since $\Gamma_\gH$ is neat and, by \cite[Remark
2.3]{uy:andre-oort}, $\gZ(\gH)(\RR)$ is compact.) We let $\Omega_T$
denote the subvariety of $\Omega$ comprising the subvarieties of
$\check{X}$ of the form $\check{X}_{\gH,x_2}$ for some
$x_2\in\check{X}_2$. By \cite[Section 3]{milne:connected}, the natural
map
\begin{align*}
  P_T\times\Omega_T\xrightarrow{\iota_{P_T}}P\times\Omega
\end{align*}
is algebraic, and it is easy to check that it is injective. We let
$\Pi(d)_T$ denote the locally closed subset
\begin{align*}
  \iota_{P_T}^{-1}(\Pi(d)^\circ)\subset P_T\times\Omega_T,
\end{align*}
and we let $\Pi(d)_{S_T}$ denote its image in
$S_T=\Gamma_\gH\backslash X_\gH$ under the natural map.

If we let $\gH_1\times\gH_2$ denote the decomposition of $\gH^\ad$
giving rise to the $\QQ$--splitting $X_\gH=X_1\times X_2$, and we let
$\Gamma_2$ denote the image of $\Gamma_\gH$ in $\gH_2(\QQ)$, we obtain
a diagram \medskip
\begin{center}
  \begin{tikzcd}
    X_2 \arrow[d, "\pi_{T,2}"] & X_\gH \arrow[d, "\pi_T"] \arrow[r, hook] \arrow[l]                         & X \arrow[d, "\pi"] \\
    \Gamma_2\backslash X_2 & S_T \arrow[l, "\phi_{T}"'] \arrow[r,
    "\iota_T"] & S.
  \end{tikzcd}
\end{center}

We let $\overline{\Pi}(d)_{S_T}$ denote the union of the Zariski
closures of the fibers of the map $\phi_{T}$ restricted to
$\Pi(d)_{S_T}$ and we let $\phi_{T,d}$ denote the map
\begin{align*}
  \overline{\Pi}(d)_{S_T}\to\phi_{T}(\Pi(d)_{S_T})
\end{align*}

\begin{lem}\label{fiberwise}
  The sets $\overline{\Pi}(d)_{S_T}$ are constructible.
\end{lem}

\begin{proof}
  It is a general fact that for any constructible map $f:X\to Y$,
  the union of the Zariski closures of the fibers is
  constructible. Since we did not find a suitable reference we give
  the details below.

  Up to taking affine covers one may assume that $X$ and $Y$ are affine. The
  union of the Zariski closures can be defined by
  \begin{equation}
    \{ x\in X : \forall_{P\in \CC[X]} [ (\forall_{x'\in f^{-1}(f(x))}  P(x')=0 ) \implies P(x)=0 ]\}.
  \end{equation}
  This would be constructible if one could restrict to quantifying
  over $P\in \CC[X]$ of degree bounded by some $N\in\N$. That is, if one
  could show that the Zariski closures of $f^{-1}(y)$ are
  set-theoretically cut out by some polynomials of uniformly bounded
  degree. Equivalently, it sufficies to show that the Zariski closures
  of these sets have uniformly bounded degrees, which is standard.
\end{proof}

We have the following structure theorem for weakly optimal
subvarieties.

\begin{teo}\label{main-const-thm}
  Let $d\in\NN$.
  \begin{itemize}
  \item[(i)] Let $W$ be a weakly optimal subvariety of $V$ such that
    $\delta_{\rm ws}(W)=d$. Then there exists an irreducible component
    $\Pi(d)^\circ$ of $\Pi(d)$ such that, if
    $T$ denotes the triple associated with $\Pi(d)^\circ$, then
    $W=\iota_{T}(W_T)$ for some irreducible component $W_T$ of a fiber
    of $\phi_{T,d}$.

  \item[(ii)] Let $\Pi(d)^\circ$ denote an irreducible component
    $\Pi(d)^\circ$ of $\Pi(d)$ and let
    $T=(X_\gH,X_1,X_2)$ denote the triple associated with
    $\Pi(d)^\circ$. If $W$ is an irreducible component of a fiber of
    $\phi_{T,d}$, then $\iota_T(W)$ is a weakly optimal subvariety of
    $V$ such that $\delta_{\rm ws}(W)=d$ and
    $\langle W\rangle_{\rm ws}=\pi(X_{\gH,x_2})$ for some
    $x_2\in X_2$.
  \end{itemize}
\end{teo}

\begin{proof}[Proof of (i)]
  Let $A$ be an irreducible analytic component of $\pi^{-1}(W)$. By
  \cite[Proposition 6.9]{DR}, $A$ is a Zariski optimal intersection
  component of $\pi^{-1}(V)$ and so, by Theorem \ref{WAS},
  $\langle A\rangle_{\rm Zar}$ is a pre-weakly special subvariety
  $X_A$ of $X$. A simple calculation shows that
  $\pi(X_A)=\langle W\rangle_{\rm ws}$ and so
  $\delta_{\rm Zar}(A)=\delta_{\rm ws}(W)= d$.

  By Lemma \ref{findarep}, for any $p=[x,1]$ with $x\in A$ satisfying
  \begin{align}\label{dimeq}
    \dim_x(\check{X}_A\cap\pi^{-1}(V))=\dim A,
  \end{align}
  we have $(p,\check{X}_A)\in\Pi(d)$. Since (\ref{dimeq}) defines an
  open subset of $A$, there exists an open subset $U$ of $A$ and an
  irreducible component $\Pi(d)^\circ$ of
  $\Pi(d)$ such that $(p,\check{X}_A)\in\Pi(d)^\circ$ for
  all $p=[x,1]$ with $x\in U$.

  Let $T=(X_\gH,X_1,X_2)$ be the triple associated with
  $\Pi(d)^\circ$. Then $\check{X}_A=\gamma\check{X}_{\gH,x_2}$ for
  some $\gamma\in\Gamma$ and some $x_2\in X_2$. Therefore,
  $(p,\gamma^{-1}\check{X}_A)\in P_T\times\Omega_T$ for all
  $p=[\gamma^{-1}x,1]$ with $x\in U$. In fact, by Proposition
  \ref{loc-closed}, these points also belong to $\Pi(d)^\circ$ as they
  belong to the $\gG(\CC)$--orbits of the points above.

  Let $W_T$ denote the irreducible component $\pi_T(\gamma^{-1}A)$ of
  $\iota_T^{-1}(W)$. Then $W_T$ is a weakly optimal subvariety of
  $V_T=\iota_T^{-1}(V)$ and $\langle W_T\rangle_{\rm ws}$ is equal to
  $S_{\gH,x_2}=\pi_T(X_{\gH,x_2})$. In particular, $W_T$ is an
  irreducible component of $S_{\gH,x_2}\cap V_T$, which is the fiber
  of $V_T\to\phi_T(V_T)$ over the point $z_2=\pi_{T,2}(x_2)$. However,
  since $\Pi(d)_{S_T}$ contains $\pi_T(\gamma^{-1}U)$, we deduce that
  $W_T$ is contained in $\overline{\Pi}(d)_{S_T}$ and, therefore, is
  an irreducible component of the fiber of $\phi_{T,d}$ over $z_2$.
\end{proof}

\begin{proof}[Proof of (ii)]
  Let $z\in \Pi(d)_{S_T}\cap W$ (observe that $\Pi(d)_{S_T}\cap W$ is
  Zariski dense in $W$ and, in particular, is non-empty). Write
  $z=\pi_{T}(x)$ for some $x=(x_1,x_2)\in X_\gH=X_1\times X_2$ (in
  particular, $W$ is an irreducible component of the fiber of
  $\phi_{T,d}$ over $\pi_{T,2}(x_2)$) and choose $(p,Y)\in\Pi(d)_T$
  lying above $z$. Clearly, we can choose $p=[x,h]$ for some
  $h\in\gH(\CC)$, and so $hY=\check{X}_{\gH,y_2}$ for some
  $y_2\in X_2$. In fact, since $x\in hY$, we conclude that $y_2=x_2$.

  Let $B$ denote an irreducible analytic component of
  $hY\cap\pi^{-1}(V)$ containing $x$ such that
  \begin{align*}
    \dim B=\dim_{x}(hY\cap\pi^{-1}(V)).
  \end{align*}
  By Lemma \ref{isweaksp}, $B$ is a Zariski optimal intersection
  component of $\pi^{-1}(V)$ such that $\delta_{\rm Zar}(B)=d$ and,
  writing $X_B$ for the pre-weakly special subvariety
  $\langle B\rangle_{\rm Zar}$ of $X$, we have
  $hY=\check{X}_{\gH,x_2}=\check{X}_B$. By \cite[Proposition 6.9]{DR},
  $W_B=\pi(B)$ is a weakly optimal subvariety of $V$, and we see that
  $\delta_{\rm ws}(W_B)=d$ and
  $\langle W_B\rangle_{\rm ws}=\pi(X_{\gH,x_2})$.

  Since $\Pi(d)_{S_T}\cap W$ is Zariski dense in $W$, we conclude that
  $\iota_T(W)$ is contained in the union of the $W_B$ as obtained
  above. This is a finite union since each $W_B$ is a weakly optimal
  subvariety of $V$ with weakly special closure $\pi(X_{\gH,x_2})$,
  and $\pi(X_{\gH,x_2})$ does not depend on $z$. Since $\iota_{T}(W)$
  is irreducible, it is contained in one such subvariety, and we
  conclude that
  \begin{align}\label{dimupperbd}
    \dim W\leq\dim B=\dim X_1-d.
  \end{align}

  On the other hand, since $\Pi(d)$ is locally closed,
  $\Pi(d)^\circ$ contains a Zariski open subset $U$ that is disjoint from the other irreducible
  components of $\Pi(d)$. Then
  $U_T=\iota_{P_T}^{-1}(U)$ is Zariski open in $\Pi(d)_T$
  and so, since it is dense and constructible, the image $U_{S_T}$ of
  $U_T$ in $\Pi(d)_{S_T}$ contains an open subset of $\Pi(d)_{S_T}$.

  Suppose then that $z\in U_{S_T}$ and $(p,Y)\in U_T$. By Lemma
  \ref{findarep}, for any $y\in B$ such that
  \begin{align*}
    \dim B=\dim_y(hY\cap\pi^{-1}(V))
  \end{align*}
  and $p=[y,1]$, we have $(p,hY)\in\Pi(d)$. Since this condition
  defines an open subset of $B$ containing $x$, we conclude that there
  exists an open subset $U_B$ of $B$ containing $x$ such that, for any
  $y\in U_B$ and $p=[y,1]$, we have $(p,hY)\in\Pi(d)_T$. It follows
  that $\pi_{T}(B)$ is contained in the fibre of $\phi_{T,d}$ over
  $\pi_{X_2}(x_2)$.

  Therefore, since $U_{S_T}$ contains an open subset of
  $\Pi(d)_{S_T}$, the irreducible components of the fibers of
  $\phi_{T,d}$ are of dimension at least $\dim B=\dim X_1-d$. Hence,
  we conclude from (\ref{dimupperbd}) that they are pure of dimension
  $\dim X_1-d$ and this concludes the proof.
\end{proof}

Observe that Theorem \ref{main-const-thm} implies \cite[Proposition
6.3]{DR}), which establishes that the weakly optimal subvarieties of
$V$ come from finitely many triples $(X_\gH,X_1,X_2)$. The novelty in
Theorem \ref{main-const-thm} is that the fibers of the $\phi_{T,d}$
are precisely the weakly optimal subvarieties of $V$ of weakly special
defect $d$.

We recall that \cite[Theorem 7.2]{DR} established that the union
$V^{\rm an}$ of the positive dimensional weakly optimal subvarieties
of $V$ of weakly special defect at most $d$ is a Zariski closed subset
of $V$. As such, either $V^{\rm an}=V$, in which case the Zilber--Pink
conjecture can be reduced to arithmetic (see \cite[Theorem 14.3]{DR}),
or its complement in $V$ is a non-empty Zariski open subset.

\section{Ingredients from differential algebraic geometry}
\label{sec:diff-algebra}

\subsection{The main statement}\label{mainstatement}

Let $\bar X$ denote a proper smooth complex algebraic variety of
dimension $d$, and let $X\subset \bar X$ denote an open dense
subset. Let $L$ denote a very ample line bundle on $\bar X$. In this
section, the \emph{degree} $\deg(Y)$ of a subvariety $Y\subset X$ is
taken to mean the degree of the Zariski closure $\bar Y\subset\bar X$
with respect to $L$.

Let $\cF$ denote a non-singular $n$--dimensional foliation of $X$.  We
further assume for simplicity that $X$ is affine and that $\cF$ is
generated by $n$ commuting regular vector fields
$\vxi=(\xi_1,\ldots,\xi_n)$. We thus may think of regular/rational
functions $P$ on $X$ as restrictions of polynomials/rational functions on
a suitable $\A^N$, and we use $\deg(P)$ to denote the (minimal) degree
of such a representative. Similarly we denote by $\deg(\xi)$ the maximum
among the degrees of the coefficients of $\xi_1,\ldots,\xi_n$ thought
of as fields $\xi_i:X\to T\A^N$.

The assumption above can always be achieved by passing to an affine
cover of $X$. For instance, identifying $X$ as an affine subvariety
of $\A^N$ and choosing generic linear coordinates $x_1,\ldots,x_N$ on
$\A^N$, there are unique rational vector fields $\xi_1,\ldots,\xi_n$
tangent to $\cF$ of the form
\begin{equation}\label{eq:xi-coeffs}
  \xi_i = \pd{}{x_i}+\sum_{j=n+1}^N c_{ij}(x)\pd{}{x_j}
\end{equation}
where $c_{ij}$ are rational functions. Moreover $[\xi_i,\xi_j]=0$ by
the Frobenius theorem. The results below can be applied, after this
affine covering process, to any foliation $\cF$ as above.

For any algebraic subvariety $V\subset X$ and $k\in\NN$ let
\begin{align*}
  \Sigma(V,\cF,k) = \{p\in V : \dim(\cL_p\cap V) \ge k\}.
\end{align*}
Our main tool is the following.

\begin{prop}\label{prop:Sigma-bound}
  Let $V\subset X$ be an algebraic subvariety. Then
  \begin{align*}
    \deg(\Sigma(V,\cF,k)) \le P_d(\deg(X),\deg(V),\deg(\vxi))
  \end{align*}
  for some explicit polynomial $P_d$ depending on $d$. Moreover the
  equations for $\Sigma(V,\cF,k)$ can be effectively computed from the
  equations defining $X,V$ and $\vxi$.
\end{prop}

\begin{rem}
  The explicit choice of affine coordinates is not strictly
  necessary to state the degree bound in
  Proposition~\ref{prop:Sigma-bound}. However, it is convenient for
  establishing the effective nature of our construction (i.e. to
  clarify the sense in which \emph{equations} for $\Sigma(V,\cF,k)$
  are to be effectively computed).
\end{rem}

We prove Proposition~\ref{prop:Sigma-bound} below after describing how
it relates to our concrete context involving flat connections on a
Shimura variety.

\subsection{Flat connections and foliations}
\label{sec:conn-vs-foliation}

Recall the situation described in Section \ref{SPBs}. Let $U$ be an
affine Zariski open subset of $S$ such that
$\pi_P^{-1}(U)\cong U\times\gG(\CC)$ and suppose that $x_1,\ldots,x_n$
is a system of \'etale co-ordinates on $U$. The connection
$\nabla\in\Omega^1(P,\mathfrak{g}_\CC)$ can be written, with respect
to these co-ordinates, as
\begin{align*}
  \nabla=\sum^n_{i=1}\Omega_idx_i,
\end{align*}
where $\Omega_i$ is an algebraic morphism
$U\to\mathfrak{g}_\CC$. Choosing a faithful representation
$\mathfrak{g}\to\mathfrak{gl}_M$, we may write the $\Omega_i$ as
matrices with entries given by polynomials in the $x_1,\ldots,x_n$.

The vector fields $\frac{\partial}{\partial x_i}$ on $U$ lift to
vector fields $\xi_i$ on $\pi_P^{-1}(U)$, which in our choice of
coordinates can be written
\begin{align*}
  \xi_i=\frac{\partial}{\partial x_i}+\Omega_i\cdot g.
\end{align*}
The vector fields $\xi_i$ commute by the flatness of $\nabla$. By
definition, their integral manifolds $\mathcal{L}_p$ at a point
$p\in P$ are given by (germs of) horizontal sections of $\nabla$.

\subsection{Multiplicity estimates}

In order to prove Proposition \ref{prop:Sigma-bound}, we will require
the following multiplicity estimate due to Gabrielov-Khovanskii
\cite{gk:noetherian}. 

\begin{teo}[\protect{\cite[Theorem~1]{gk:noetherian}}]\label{thm:gk-mult}
  Let $P=(P_1,\ldots,P_n)\in\OOO(X)^n$ and let $p\in X$. Suppose that
  the restriction of $P$ to the leaf ${\cL_p}$ of $\cF$
  through $p$ has an isolated zero at $p$. Then
  \begin{align*}
    \mult_p P\rest{\cL_p} < f_d(\deg(\vxi), \deg(P)),
  \end{align*}
  for some explicit polynomial $f_d$.
\end{teo}

We will use this result to characterize the locus of points where the
intersection of $\cL_p$ with the vanishing locus of $P$ is positive-dimensional. For this purpose, we are interested in expressing the
condition that a tuple of functions admits a common zero of
multiplicity at least $k$ by means of differential algebraic
conditions.

Let $F=(F_1,\ldots,F_n)$ denote an $n$-tuple of holomorphic functions
in some domain $\Omega\subset\CC^n$. The problem above is addressed in
\cite{me:mult-ops} by means of a collection $\{M^\alpha\}$ of
``multiplicity operators'' of order $k$. These are polynomial partial
differential operators of order $k$, i.e. polynomial combinations of
$F_1,\ldots,F_n$ and their first $k$ derivatives. We will usually
denote a multiplicity operator of order $k$ by $\mo$ and write
$\mo_pF$ for $[\mo(F)](p)$.

\begin{prop}[\protect{\cite[Proposition~5]{me:mult-ops}}]\label{prop:mo-basic}
  We have $\mult_p F>k$ if and only if $\mo_pF=0$ for all multiplicity
  operators of order $k$.
\end{prop}

For every $p\in X$, we let $\phi_p:B\to\cL_p$ denote germ of a
holomorphic map, for some open ball $B\subset\CC^n$ centered at the
origin, satisfying $\partial\phi_p/\partial x_i=\xi_i$ for
$i=1,\ldots,n$. We refer to this map as the $\vxi_i$--chart on
$\cL_p$. When $P=(P_1,\ldots,P_n)\in\OOO(X)^n$ we may apply the
multiplicity operator $\mo$ to $P$ by evaluating the derivatives along
$\xi_{1},\ldots,\xi_{n}$, which amounts to computing, for each point
$p\in X$, the multiplicity operator of $P\rest\cL_p$ in the
$\vxi_i$--chart.

\begin{lem}\label{lem:mo-complexity}
  For any multiplicity operator $\mo$ of order $k$ we have
  \begin{align*}
    \deg(\mo P) \le f_d(\deg(P),\deg(\vxi),k)
  \end{align*}
  for some explicit polynomial $f$.
\end{lem}

\begin{proof}
  This follows easily since $\mo$ is defined by expanding a
  determinant, of size polynomial in $k$, with entries defined in
  terms of $P$ and its $\vxi_i$--derivatives up to order $k$.
\end{proof}

\subsection{Proof of Proposition~\ref{prop:Sigma-bound}}

Recall the situation described in Section \ref{mainstatement}. It is
classical that $V\subset X$ is then cut out by polynomials of degree
at most $\deg(V)$. Denote the set of these polynomials by $\cP$. Then
we have
\begin{align*}
  \Sigma(V,\cF,k) = \bigcap_{\substack{\cP'\subset\cP\\ \#\cP'=n-k+1}} \Sigma(V(\cP'),\cF,k).
\end{align*}
Indeed, the inclusion $\subset$ is obvious. For the other inclusion,
suppose $p\not\in\Sigma(V,\cF,k)$ so that $\dim(V\cap\cL_p)<k$. In
particular $\cL_p\not\subset V$, so there exists an equation
$P_1\in\cP$ not identically zero on $\cL_p$. If $n-k>1$ then,
similarly, no component of the intersection of $\cL_p$ with the
vanishing locus of $P_1$ is contained in $V$, so there exists $P_2$
not vanishing on any of these components. Reiterating $n-k$ steps of
this form, we obtain $\cP'=\{P_1,\ldots,P_{n-k+1}\}$ with
$\dim(V(\cP')\cap\cL_p)=k-1$, so $p\not\in\Sigma(V(\cP'),\cF,k)$.

By Bezout's Theorem, $\deg(\Sigma(V,\cF,k))$ is bounded by a
polynomial in the maximum of the $\deg(\Sigma(V(\cP'),\cF,k))$ for
$\cP'$ as above. Hence, it is enough to prove
Proposition~\ref{prop:Sigma-bound} assuming that $V$ is a complete
intersection defined by equations $P=(P_1,\ldots,P_{n-k+1})$.

We now make a similar reduction involving the foliation $\cF$. Namely,
\begin{align}\label{eq:SigmaFvSigmaF'}
  \Sigma(V,\cF,k) = \bigcap_{\substack{\cF'\subset\cF\\ \dim\cF'=n-k+1}} \Sigma(V,\cF',1),
\end{align}
where the intersection is taken over foliations $\cF'$ generated by
linear combinations of $n-k+1$ of the vector fields comprising
$\vxi$. Again the inclusion $\subset$ is obvious. For the other
inclusion, suppose $p\not\in\Sigma(V,\cF,k)$, so that
$\dim(V\cap\cL_p)<k$. Intersecting with linear hyperplanes passing
through the origin in the $\vxi$--chart on $\cL_p$, we find, similarly
to the previous step, $k-1$ such hyperplanes defining a subleaf
$\cL_p'$ with $\dim(V\cap\cL_p')=0$. Noting that $\cL_p'$ is a leaf of
a subfoliation $\cF'$ as above finishes the proof.

By Bezout's Theorem, as above, it suffices, replacing $\cF$ by the
$\cF'$, to prove Proposition~\ref{prop:Sigma-bound} in the case $k=1$.
In this case we have
\begin{align*}
  \Sigma(V,\cF,1) &= \{ p\in V : \dim(\cL_p\cap V)\ge1\} \\&= \{p\in V:\mult_p P\rest{\cL_p}=\infty\}
                                                             = \{ p\in V:\mult_p P_\rest{\cL_p}\ge\mu \},
\end{align*}
where $\mu$ is the multiplicity bound of
Theorem~\ref{thm:gk-mult}. Finally, according to
Proposition~\ref{prop:mo-basic}, the right-hand side is the zero locus
of all multiplicity operators $\mo[\mu]_pP$ taken with respect to the
foliation $\cF$. The degrees of all of these polynomials are bounded
by Lemma~\ref{lem:mo-complexity}. Applying Bezout's Theorem concludes
the proof of the degree bound.

Finally, we indicate how to effectively obtain a system of equations
for $\Sigma(V,\cF,k)$. The only step above which isn't effective
a priori is (\ref{eq:SigmaFvSigmaF'}), where one intersects an
infinite collection of equations. To deal with this, we first note
that it would suffice to consider $\cF'$ in some open-dense subset of
the Grassmannian of $n-k+1$-dimensional subsets of the span of
$\vxi$. We can generate such $\cF'(c)$ as the span of
$\xi'_1(c),\ldots,\xi'_{n-k+1}(c)$ with
\begin{equation}
  \xi'_i = \xi_i+\sum_{j=n-k+2}^n c_{ij}\xi_j
\end{equation}
and $(c_{ij})\in\A^{(n-k+1)(k-1)}$. Repeating the construction above
with $\cF'(c)$ and treating $c_{ij}$ as independent variables, we
obtain a system of equations $E_1(x,c)=\cdots=E_Q(x,c)=0$ such that
$x\in\Sigma(V,\cF,k)$ if and only if the equations vanish at $(x,c)$
for every $c$. It is then clear that $\Sigma(V,\cF,k)$ is cut out by
the coefficients of $E_1,\ldots,E_Q$ viewed as polynomials in the
$c$-variables.

\section{Effective bounds for degrees of weakly optimal subvarieties}

Recall the situation described in Section \ref{mainconst}. For $i=1,\ldots,n$, let $U_i$ denote a Zariski open subvariety of $S$ such
that $\cup_{i=1}^nU_i=S$ and $\pi_P^{-1}(U_i)\cong U_i\times\gG(\CC)$
is trivial. Furthermore, we can and do assume that, under the
embedding $S\to\PP^N$ given by $L^{\otimes k_\Gamma}_\Gamma$, each $U_i$ is contained in
one of the standard affine charts. Similarly, let $\Omega_i$ denote
Zariski open subvarieties of $\Omega$ (already considered as a
subvariety of some $\PP^{M}$), each contained in a standard affine
chart, such that $\Omega=\cup_{i=1}^n\Omega_i$. Fix an embedding of
$\gG$ into some affine space $\AAA^L$. For $i,j=1,\ldots,n$, let
$U_{ij}$ denote the Zariski open subset
$U_i\times\gG(\CC)\times \Omega_j$ of $P\times\Omega$. The following is an immediate corollary of Proposition \ref{loc-closed}.

\begin{lem}\label{deginput}
  Let $d\in\NN$ and let $\Pi(d)^{\circ}$ denote an irreducible
  component of $\Pi(d)$. For any
  $i,j\in\{1,\ldots,n\}$, the degree of the Zariski closure of the
  locally closed subset $U_{ij}\cap\Pi(d)^\circ$ of
  $P\times\Omega$, considered as a subvariety of $\AAA^{N+L+M}$, is bounded by $f_S(\deg_{L^{\otimes k_\Gamma}_\Gamma}(V))$ for some polynomial $f_S$ depending only on $S$.
\end{lem}

The main result of this section is the following.

\begin{teo}\label{degreebound}
  Let $d\in\NN$ and let $W$ be a weakly optimal subvariety of $V$ such
  that $\delta_{\rm ws}(W)=d$. Then 
  \begin{align*}
\deg_{L^{\otimes k_\Gamma}_\Gamma}(W)\leq f_S(\deg_{L^{\otimes k_\Gamma}_\Gamma}(V)).
\end{align*}
\end{teo}

\begin{proof}[Proof of Theorem \ref{degreebound}]
  Since $k_\Gamma$ depends only on $S$, we may replace $L^{\otimes k_\Gamma}_\Gamma$ with $L_\Gamma$, after possibly replacing $f_S$.
  
  Recall the situation reached at the end of the proof of Theorem
  \ref{main-const-thm} (i). Let $\Pi(d)_S$ denote the image of
  $\Pi(d)^\circ$ in $S$ and let $V_0$ denote its Zariski
  closure. Let $V_{T,0}$ denote the Zariski closure of $\Pi(d)_{S_T}$
  in $S_T$. Observe that $\phi_{T}(V_{T,0})$ is closed as $\phi_{T}$
  extends to a (projective) morphism on the Baily--Borel
  compactifications.

  We will prove the theorem in a succession of lemmas. The following
  lemma will be used in the proofs of Lemma \ref{obs} (ii) and (iii).

\begin{lem}\label{irred0}
  The subvariety $V_{T,0}$ is irreducible and $\iota_T(V_{T,0})=V_0$.
\end{lem}

\begin{proof}
  We appeal to the commutative diagram of natural morphisms
  \begin{center}
    \begin{tikzcd}
      P_T\times\Omega_T \arrow[d, "f_T" description] \arrow[r] \arrow[rr, "\iota_{P_T}" description, bend left] & \iota_T^*(P\times\Omega) \arrow[ld, "g_T" description] \arrow[r] & P\times\Omega \arrow[d, "f" description] \\
      S_T \arrow[rr, "\iota_T" description] & & S,
    \end{tikzcd}
  \end{center}
  where $\iota_T^*(P\times\Omega)$ is the pullback of $P\times\Omega$
  to $S_T$. Recall that $\iota_{P_T}$ is injective.

  First observe that the $\gG(\CC)$--orbit of (the image of)
  $\Pi(d)_T$ in $P\times\Omega$ is $\Pi(d)^\circ$. Indeed,
  the left-right inclusion is clear from Proposition
  \ref{loc-closed}. Therefore, let $(p,Y)\in\Pi(d)^\circ$
  and write $p=[x,g]$. Then $gY=\gamma\check{X}_{\gH,x_2}$ for some
  $\gamma\in\Gamma$ and some $x_2\in X_2$. We rewrite
  $p=[\gamma^{-1}x,\gamma^{-1}g]$ and we see that
  $\gamma^{-1}g(p,Y)\in P_T\times\Omega_T$, which proves the claim.

  Let $Z_T$ denote the Zariski closure of (the image of) $\Pi(d)_T$ in
  $\iota_T^*(P\times\Omega)$. Then, by the preceding paragraph, (the
  image of) $\gG(\CC)Z_T$ is a Zariski dense subset of
  $\Pi(d)^\circ$. Let $Y_T$ denote an irreducible component of $Z_T$
  such that $\gG(\CC)Y_T$ is also dense in $\Pi(d)^\circ$. Then, since
  the Zariski closure of $f(\Pi(d)^\circ)$ is $V_0$, we see that
  $f(\gG(\CC)Y_T)=f(Y_T)$ is dense in $V_0$ and, by the commutativity
  of the diagram, equal to $\iota_T(Y_{S_T})$, where
  $Y_{S_T}=g_T(Y_T)$.

  Now let $Y'_T$ denote another irreducible component of $Z_T$. We
  claim that $g_T(Y'_T)$ is contained in the Zariski closure of
  $Y_{S_T}$. To see this, observe that $Y'_T$ is contained in the
  Zariski closure of $\gG(\CC)Y_T$ in
  $\iota_T^*(P\times\Omega)$. Therefore, $g_T(Y'_T)$ is contained in
  the Zariski closure of $g_T(\gG(\CC)Y_T)=g_T(Y_T)$.

  It follows that the Zariski closure of $Y_{S_T}$ is equal to
  $V_{T,0}$. Hence, $V_{T,0}$ is irreducible. The fact that
  $\iota_T(V_{T,0})=V_0$ now follows easily from the facts that
  $\iota_T$ is closed and $V_0$ is equal to the Zariski closure of
  $f(\Pi(d)^\circ)$.
\end{proof}

The following lemma will reduce the proof to Lemma \ref{ineq}.

\begin{lem}\label{obs}
  \

  \begin{itemize}
  \item[(i)] $\deg_{L_\Gamma}(V_0)\leq f_S(\deg_{L_\Gamma}(V))$;
  \item[(ii)]
    $\deg_{L_{\Gamma_\gH}}(V_{T,0})\leq r^3_\gG\deg_{L_\Gamma}(V_0)$,
    where $r_\gG$ denotes the rank of $\gG$;
  \item[(iii)]
    $\deg_{L_\Gamma}(W)\leq r^3_\gG\deg_{L_{\Gamma_\gH}}(W_T)$.
  \end{itemize}
\end{lem}

\begin{proof}
  \

  \begin{itemize}
  \item[(i)] There exist $i,j\in\{1,\ldots,n\}$ such that the Zariski
    closure of $U_{ij}\cap\Pi(d)^\circ$ dominates
    $V_0$. Therefore, the result follows easily from Lemma
    \ref{deginput}.
    % \footnote{See, for example, the user Angelo's
    %   answer to the Mathoverflow question
    %   https://mathoverflow.net/questions/63451/degree-of-image-of-a-polynomial-map}
  \item[(ii)] This follows immediately from Lemma \ref{irred0},
    \cite[Lemma 4.1]{DJK} (a corollary of \cite[Proposition
    5.3.10]{KY:AO}) and \cite[Lemma 4.2]{DJK}.
  \item[(iii)] See (ii).
  \end{itemize}
\end{proof}

The following lemma will be used in the proof of Lemma \ref{ineq}.

\begin{lem}\label{obs0}
  \

  \begin{itemize}
  \item[(i)] $W_T$ is an irreducible component of the fiber of
    $V_{T,0}\to \phi_{T}(V_{T,0})$ over $z_2$;
  \item[(ii)] $\dim W_T=\dim X_1-d$ is the generic dimension of the
    fibers of $V_{T,0}\to \phi_{T}(V_{T,0})$;
  \end{itemize}
\end{lem}

\begin{proof}
  \

  \begin{itemize}
  \item[(i)] Observe that $V_{T,0}$ is contained in $V_T$ and $W_T$ is
    contained in $V_{T,0}$. Therefore, since $W_T$ is an irreducible
    component of the fiber of $V_T\to\phi_{T}(V_T)$ over $z_2$, the
    claim follows.

  \item[(ii)] First observe that the Zariski closures of
    $\Pi(d)_{S_T}$ and $\overline{\Pi}(d)_{S_T}$ coincide. That is,
    they are both $V_{T,0}$. Also, $\phi_{T}(V_{T,0})$ is contained in
    the Zariski closure of $\phi_{T}(\Pi(d)_{S_T})$. However,
    $\phi_{T}(V_{T,0})$ is closed and contains
    $\phi_{T}(\Pi(d)_{S_T})$. Therefore, $\phi_{T}(V_{T,0})$ is equal
    to the Zariski closure of $\phi_{T}(\Pi(d)_{S_T})$. Since, by
    Theorem \ref{main-const-thm}, the fibers of $\phi_{T,d}$ are pure
    of dimension $\dim X_1-d$, the claim follows.
  \end{itemize}
\end{proof}

It remains to prove the following lemma.

\begin{lem}\label{ineq}
  We have
  \begin{align*}
    \deg_{L_{\Gamma_\gH}}(W_T)\leq\deg_{L_{\Gamma_\gH}}(V_{T,0}).
  \end{align*}
\end{lem}

\begin{proof}
  In order to simplify notation, we will, for the remainder of the
  proof, replace $\Gamma_\gH$ with $\Gamma$, $S_T$ with $S$, $V_{T,0}$
  with $V$, $W_T$ with $W$, and $\dim W$ with $d$ (as opposed to
  $\dim X_1-d$). We will also reassign $\iota$ to be the (proper)
  closed embedding $V\rightarrow S$. Finally, for $i=1,2$, we let
  $\Gamma_i$ denote the image of $\Gamma$ in $\gH_i(\QQ)$, we let $f$
  denote the natural morphism $S\rightarrow S_1\times S_2$, where
  $S_i=\Gamma_i\backslash X_i$ is the Shimura variety associated with
  $(\gH_i,X_i)$ and $\Gamma_i$, and we let $\phi_i$ denote the
  projection $S\rightarrow S_i$ (which factors as $f$ composed with
  the natural projection $f_i$ from $S_1\times S_2$ to $S_i$).

  By the projection formula, the degree of
  $c_1(L_\Gamma)^d\cap[\iota(V)]$ is equal to the degree of
  $c_1(\iota^*L_\Gamma)^d\cap[V]\in A_0(V)$. Let
  $V_2=\phi_2(\iota(V))$, which, as explained above, is closed. There
  exists a Zariski open subset $U_2\subset V_2$ such that, if
  $U=(\phi_2\circ\iota)^{-1}(U_2)$, then $(\phi_2\circ\iota)_{|U}$ is
  flat. In particular, its fibers are equidimensional, of dimension
  $d$ (by Lemma \ref{obs0} (ii)).

  Now consider the excision exact sequence
  \begin{align*}
    A_0(V\setminus U)\xrightarrow{i_*} A_0(V)\xrightarrow{j^*} A_0(U),
  \end{align*}
  where $i_*$ is the pushforward associated with the (proper) closed
  embedding $i:V\setminus U\rightarrow V$, and $j^*$ is the pullback
  associated with the (flat) inclusion $j:U\rightarrow V$. We see that
  the degree of $c_1(\iota^*L_\Gamma)^{\dim V}\cap[V]$ is at least the
  degree of
  \begin{align*}
    j^*(c_1(\iota^*L_\Gamma)^{\dim V}\cap[V])=c_1(j^*\iota^*L_\Gamma)^{\dim V}\cap[U]\in A_0(U).
  \end{align*}	

  Next, as in \cite[Proposition 5.3.2 (1)]{KY:AO}, we have
  \begin{align*}
    L_\Gamma=f^*(L_{\Gamma_1}\boxtimes L_{\Gamma_1})=f^*(f_1^*L_{\Gamma_1}\otimes f_2^*L_{\Gamma_2})=\phi_1^*L_{\Gamma_1}\otimes\phi_2^*L_{\Gamma_2}.
  \end{align*}
  Therefore,
  \begin{align*}
    j^*\iota^*L_\Gamma=j^*\iota^*\phi_1^*L_{\Gamma_1}\otimes j^*\iota^*\phi_2^*L_{\Gamma_2},
  \end{align*}
  and so
  \begin{align*}
    c_1(j^*\iota^*L_\Gamma)^{\dim V}\cap[U]=\sum_{r=0}^{\dim V}\binom{\dim V}{r}\left[c_1(j^*\iota^*\phi_1^*L_{\Gamma_1})^r\cap\left(c_1(j^*\iota^*\phi_2^*L_{\Gamma_2})^{\dim V-r}\cap[U]\right)\right].
  \end{align*}
  Since $\phi_2\circ\iota\circ j$ is flat,
  \begin{align*}
    c_1(j^*\iota^*\phi_2^*L_{\Gamma_2})^{\dim V-r}\cap[U]=j^*\iota^*\phi^*_2(c_1(L_{\Gamma_2})^{\dim V-r}\cap[U_2])\in A_r(U),
  \end{align*}
  where $j^*\iota^*\phi^*_2$ is the flat pull-back
  $A_0(U_2)\rightarrow A_r(U)$. Therefore, since $L_{\Gamma_2}$ is
  ample, these classes can be represented by non-negative sums of
  $r$--cycles (which are zero if $r>d$). Hence, since $L_{\Gamma_1}$
  is ample, each summand of the above sum can be represented by a
  non-negative sum of $0$--cycles.  In particular,
  $c_1(j^*\iota^*\phi_2^*L_{\Gamma_2})^{\dim V-d}\cap[U]$ can be
  represented by the cycle associated to finitely many fibers of
  $(\phi_2\circ\iota\circ j)\rest{U}$. For such a fiber $F$, we have
  \begin{align*}
    c_1(j^*\iota^*\phi_1^*L_{\Gamma_1})^d\cap [F]=c_1(j^*\iota^*L_{\Gamma})^d\cap [F]=\deg_{j^*\iota^*L_{\Gamma}}(F)=\deg_{L_\Gamma}(\iota(j(F))),
  \end{align*}
  where the first equality can be deduced from a binomial expression
  as above. Therefore, since $(\phi_2\circ\iota\circ j)\rest{U}$ is
  flat, and all fibers of a flat family of subschemes of a projective
  space have the same Hilbert polynomial, we conclude that, if $W$ is
  an irreducible component of $(\phi_2\circ\iota\circ j)^{-1}(z)$ for
  some $z\in U_2$, then
  \begin{align*}
    \deg_{L_\Gamma}(W)\leq\deg_{L_\Gamma}(V),
  \end{align*}
  as claimed.

  Therefore, by Lemma \ref{obs0} (i), it remains to deal with the case
  when $W$ is an irreducible component of
  $(\phi_{2}\circ\iota)^{-1}(z)$ for some $z\in V_2\setminus U_2$. To
  that end, let $C$ be an irreducible algebraic curve in $V_2$ passing
  through $z$ such that $C\cap U_2\neq\emptyset$ (just choose a point
  in $U_2$ and use the fact that, for any two points $x$ and $y$ in an
  irreducible algebraic variety $Z$, there exists an irreducible
  algebraic curve $C\subset Z$ such that $x,y\in C$). Let $Y$ denote
  an irreducible component of $(\phi_{2}\circ\iota)^{-1}(C)$
  containing $W$. Note that $W\subsetneq Y$ as
  \begin{align*}
    \dim Y\geq\dim V-\dim V_2+1=d+1>\dim W.
  \end{align*}

  Therefore, the morphism $Y\rightarrow C$ is dominant. Let
  $\tilde{C}$ denote the normalization of $C$, and let $\tilde{Y}$
  denote the fiber product $Y\times_C\tilde{C}$. Note that $\tilde{Y}$
  is irreducible and $\tilde{Y}\rightarrow\tilde{C}$ is
  dominant. Since $\tilde{C}$ is regular, it follows from \cite[Prop
  9.7]{hartshorne1977algebraic} that $\tilde{Y}\rightarrow\tilde{C}$
  is flat. As such, its fibers all have the same degree with respect
  to $\eta^*L_\Gamma$.

  Let $\eta:\tilde{Y}\rightarrow Y$ denote the natural map. It is
  finite and surjective. Therefore, if we let $\tilde{W}$ denote an
  irreducible component of $\eta^{-1}(W)$, we have
  \begin{align*}
    \deg_{L_\Gamma}(W)\leq\deg_{\eta^*L_\Gamma}(\tilde{W})
  \end{align*}
  and, since $\tilde{W}$ is a fiber of $\tilde{Y}\rightarrow\tilde{C}$
  and $C$ passes through $U_2$, the claim follows.
\end{proof}
The theorem now follows immediately, combining Lemma \ref{obs}
(i)--(iii) and Lemma \ref{ineq}.
\end{proof}

\section{Effective determination of the weakly optimal locus}
\label{sec:fully-effective}

Recall the setup described in
Section~\ref{sec:intro-fully-effective}. This section is devoted to
the proof of Theorem~\ref{thm:effective-computation}.

Throughout the proof, we think of the family $T$ as embedded in
projective space with respect to the prescribed embedding.  We begin
by computing an affine cover $\{V_\alpha\}$ of $V$, such that over
each $V_\alpha$ once can select an explicit set of sections
$\omega_1,\ldots,\omega_g$ for the sheaf of relative differentials
$\Omega^1_{T/V}$ and an additional set of meromorphic differentials of
the second kind (i.e. with vanishing residues)
$\omega_{g+1},\ldots,\omega_{2g}$ for the sheaf
$\Omega^1_{T/V}(N\cdot D)$ where $D$ is a hyperplane divisor and
$n\gg1$, such that that $\omega_1,\ldots,\omega_{2g}$ are pointwise
linearly independent everywhere. Such a choice of differentials can be
computed explicitly for (families of) algebraic curves by classical
methods. For the holomorphic differentials, see
e.g. \cite[Theorem~9.3.1]{bk:plane-curves} where it is attributed to
Abel and Riemann. For meromorphic differentials of the 2nd type see
e.g. \cite[Propositions~9.3.8, 9.3.9]{bk:plane-curves}. Below we
continue with $V$ replaced by one of the $V_\alpha$, and assume that
the sections above are pointwise linearly independent over $V$.

Having computed a base for the de-Rham cohomology $H^1(T/V)$ we may
now compute the Gauss--Manin connection
$\nabla:H^1(T/V)\to H^1(T/V)\otimes\Omega^1_V$ explicitly. The fact
that the Gauss--Manin connection admits a purely algebraic construction
is essentially due to Manin \cite{manin:kernel}. The approach of Manin
is fully explicit and there is no difficulty in principle carrying it
out computationally.

The sections $\omega_1,\ldots,\omega_{2g}$ provide a trivialization of
$H^1(T/V)$. Thinking of $V\times\gGL_{2g}(\CC)$ as a principal
$\gGL_{2g}(\CC)$--bundle with respect to the action $g(v,\Pi)=(v,\Pi g^{-1})$, we
may express $\nabla$ as a flat connection on this trivial bundle as
follows
\begin{equation}
  d\Pi = \Omega \cdot \Pi, \qquad \Omega\in\mathfrak{gl}_n(\Lambda^1_V).
\end{equation}
In fact, the construction that follows could be expressed in terms of
this $\gGL_{2g}(\CC)$--connection. However, to stress the relationship with
the general formalism of Shimura varieties considered in the first
part of the paper, we show that one can explicitly compute a
$\gGSp_{2g}(\CC)$--bundle $P\subset V\times\gGL_{2g}(\CC)$ compatible with $\nabla$.

The existence of such a bundle $P$ follows from the fact that $\nabla$
preserves the symplectic form on $H^1(T/V)$ induced by duality from
the intersection form on $H_1(T/V,\ZZ)$. More explicitly, let
$\delta_1(v),\ldots,\delta_{2g}(v)$ denote an (eventually multivalued)
choice of symplectic basis of $H_1(T/V,\ZZ)$, i.e. with the
intersection form $(\delta_i,\delta_j)$ given by
\begin{equation}
  J = \begin{pmatrix}
    0 & I_g \\
    -I_g & 0
  \end{pmatrix}.
\end{equation}
Then
\begin{equation}\label{eq:Pi-def}
  \Pi(v) =
  \begin{pmatrix}
    \oint_{\delta_1(v)}\omega_1 & \cdots & \oint_{\delta_{2g}(v)}\omega_1 \\
    \vdots & \ddots & \vdots \\
    \oint_{\delta_1(v)}\omega_{2g} & \cdots & \oint_{\delta_{2g}(v)}\omega_{2g}
  \end{pmatrix} =
  \begin{pmatrix}
    A(v) & B(v) \\ C(v) & D(v)
  \end{pmatrix}
\end{equation}
is a section of $\nabla$ and $\Pi J \Pi^T = \Lambda$ defines a
regular mapping from $V$ to $\gGL_{2g}(\CC)$ (single-valuedness follows from
that fact that the monodromy of $\Pi(v)$ respects the intersection
form, and regularity then follows from GAGA). Thus the subset of
$V\times\gGL_{2g}(\CC)$ defined by $\Pi J\Pi^T=\Lambda$ is a
$\nabla$--invariant principal $\gSp_{2g}(\CC)$--bundle with respect to the
action
\begin{equation}
  g\cdot (v,\Pi) = (v,\Pi\cdot g^{-1}) \quad g\in\gSp_{2g}(\CC).
\end{equation}
However, as we will see
below, this bundle is not defined over a number field, and we will show instead how to construct the corresponding
$\gGSp_{2g}(\CC)$--bundle explicitly over a number field.

By definition, $\Lambda(v)_{ij}=(\omega_i,\omega_j)_v$ is the matrix
representing the symplectic form on $H^1(X_v)$. The explicit
computation of this form reduces to the following bilinear relation
for meromorphic differentials of the second kind.

\begin{lem}
  Let $\omega,\eta$ be two meromorphic differentials of the second
  kind. Then
  \begin{equation}\label{eq:bilinear-recip}
    \frac{1}{2\pi i} \left( \sum_{j=1}^g \oint_{\delta_j}\omega\oint_{\delta_{j+g}}\eta -
      \oint_{\delta_{j+g}}\omega\oint_{\delta_{j}}\eta \right) =
    \sum_{P} {\rm res}_P (f\eta)
  \end{equation}
  where $P$ ranges over the poles of $\omega$ and $\eta$, and $f$ is
  any primitive of $\omega$. Note that since $\eta$ has no residues,
  ${\rm res}_P(f\eta)$ is independent of the choice of the primitive.
\end{lem}

The left hand side of~\eqref{eq:bilinear-recip} is, by definition, the
symplectic pairing $(\omega,\eta)$ up to the constant $2\pi i$. The
right hand side can be explicitly computed in local coordinates around
$P\in V$, i.e. it depends only on finitely many Laurent coefficients
of $\omega,\eta$ in local coordinates around the poles. Using this, one
may explicitly compute each entry of $2\pi \Lambda(v)$ as a regular
function on $V$.

Having computed $2\pi\Lambda(v)$, we further simplify the computation as
follows. The Riemann bilinear relations imply that
$\omega_1,\ldots,\omega_g$ span an isotropic space, and $\Lambda(v)$
defines a non-degenerate pairing between this space and the span of
$\omega_{g+1},\ldots,\omega_{2g}$. By elementary linear algebra, one
may now replace each section $\omega_{g+j}$ by a linear combination
$\omega_{g+j}'$ such that $\omega_1,\ldots,\omega_g$ and
$\omega'_{g+1},\ldots,\omega'_{2g}$ form a standard symplectic
basis. Assume without loss of generality that we have made such a
choice, so that $\Lambda(v)\equiv 2\pi J$. With this choice, $\nabla$
restricts to a connection on the trivial bundle $P=V\times\gGSp_{2g}(\CC)$
with the connection equation
\begin{equation}
  d \Pi=\Omega\cdot\Pi, \qquad \Omega\in\mathfrak{sp}_{2g}(\Lambda^1_V)
\end{equation}
and the left $\gGSp_{2g}(\CC)$-action given as before by
$g(s,\Pi)=(s,\Pi g^{-1})$.

Denote by $X=\mathcal{H}_g$ the Siegel upper half-space, by $\check X$
the compact dual, and by $X'\subset\check X$ the set of symmetric
$g\times g$ matrices. We have a $\gGSp_{2g}(\CC)$--equivariant rational map
$\beta:P\to X'$ given by $\beta(v,\Pi)=B^{-1}A$ where $A$ and $B$ are the
blocks given in~\eqref{eq:Pi-def}. One can verify that, since
$\Pi\in \gGSp_{2g}(\CC)$, the image of $\beta$, when defined, is an element
of $X'$. The map $\beta$ extends to a regular map
$\beta:P\to\check X$.

Recall that we denote by $ f:\tilde V\to V$ an \'etale cover,
$\tilde T\to\tilde V$ the base change of $T$ by $f$, and choose $f$
such that $\tilde T$ is compatible with an $N$-level structure (say
for $N=3$). We denote by $\iota:\tilde V\to\cA_{g,N}$ the
corresponding moduli map.

\begin{prop}\label{prop:connection-compare}
  We have $f^*(P,\nabla)\simeq  \iota^*(\cP,\nabla_0)$ where $\cP$
  denotes the canonical bundle on $\mathcal{A}_{g,N}$ and $\nabla_0$
  its canonical connection.
\end{prop}
\begin{proof}
  By functioriality of the Gauss--Manin connection, if $\tilde\nabla$
  denotes the connection on $\tilde P\to\tilde V$ then we have
  $\tilde\nabla=f^*\nabla$. It will therefore suffice to prove that
  $(\tilde P,\tilde\nabla)\simeq\iota^*(\cP,\nabla_0)$. Thus, we may
  assume, without loss of generality, that the family $T\to V$
  already respects the $N$--level structure and that we have a map
  $\iota:V\to S=\cA_{g,N}$. Denote by $\Gamma\subset\gGSp_{2g}(\ZZ)$ the
  neat subgroup corresponding to the $N$-level structure.

  Choose a generic $v_0\in V$ and consider the period map $\Pi(v)$
  defined by~\eqref{eq:Pi-def} around $v_0$. Since
  $\delta_1,\ldots,\delta_{2g}$ form a sympectic basis, the Riemann
  bilinear relations imply that $\Pi(v)\in X$ globally (i.e. after
  arbitrary analytic continuation). By definition of the moduli
  interpretation, $\iota$ is given by
  \begin{equation}\label{eq:iota-def}
    \iota(v) = \pi(\beta(v,\Pi(v))), \qquad \pi:X\to\Gamma\backslash X.
  \end{equation}
  Indeed, $\Pi(v)$ is the period matrix of the fiber $T_v$. Hence, the
  Jacobian of $T_v$ is given by the lattice spanned by the columns of
  $(A\ B)$, and $\beta(v,\Pi(v))=B^{-1}(v)A(v)\in X$ is the point
  representing this Jacobian in $X$. In particular, write
  $x_0=\beta(v_0,\Pi(v_0))$ so that $\pi(x_0)=\iota(v_0)$.

  We will show that $(P,\nabla)\simeq\iota^*(\cP,\nabla_0)$ by showing
  that they define the same $\Gamma$--representation of
  $\pi_0(V,v_0)$. Let $\gamma\in\pi_0(V,v_0)$ be a closed loop, and
  let $\gamma_X\subset X$ be the curve obtained by lifting $\gamma$ to
  $X$. Then the endpoint of $\gamma_X$ is a point $g\cdot x_0$ for
  some $g\in\Gamma$. According to~\eqref{eq:iota-def} and the
  equivariance of $\beta$, the monodromy of $\nabla$ along $\gamma$ is
  $g$, as it is the unique element of $\Gamma$ mapping $x_0$ to
  $g\cdot x_0$. On the other hand, the monodromy of $\iota^*\nabla_0$
  along $\gamma$ is the monodromy of $\nabla_0$ along $\iota(\gamma)$,
  which equals $g$ for the same reason. This shows that the
  representations are indeed the same.
\end{proof}

One can repeat the proof of Proposition~\ref{loc-closed} with the
bundle $(P,\nabla)$ in place of $(\cP,\nabla_0)$ to define sets
$\Pi'(d)\subset P$. By Proposition~\ref{prop:connection-compare} it
follows that $f^*\Pi'(d)=\iota^*\Pi(d)$ over $f^{-1}(V)$ as claimed.

\bibliographystyle{abbrv}
\bibliography{basic}

\end{document}